\documentclass[11pt]{amsart}
\usepackage{amsmath,amssymb,amsthm,amsfonts}
\usepackage[alphabetic]{amsrefs}
\usepackage{mathrsfs}
\usepackage[arrow,matrix,curve,cmtip,ps]{xy}
\usepackage{graphicx,tikz}
\usepackage[hypertexnames=true, citecolor = green, colorlinks = false, linkcolor = red, linktoc = none, pdffitwindow = false, urlbordercolor = white]{hyperref} 
\usepackage[shortlabels]{enumitem}
\usepackage[final]{pdfpages}
\usepackage{thmtools}
\usepackage{cleveref}
\usepackage{multirow}

\allowdisplaybreaks

 \newcommand{\IF}[0]{\mathbb{F}}

\newcommand{\IK}[0]{\mathbb{K}}

\newcommand{\IQ}[0]{\mathbb{Q}} 
 \newcommand{\IT}[0]{\mathbb{T}}

\newcommand{\C}{\ensuremath{\mathbb{C}}}   

 \newcommand{\CD}[0]{\mathcal{D}}
 \newcommand{\CF}[0]{\mathcal{F}}
\newcommand{\CG}[0]{\mathcal{G}} \newcommand{\CH}[0]{\mathcal{H}}
 
\newcommand{\CK}[0]{\mathcal{K}} \newcommand{\CL}[0]{\mathcal{L}}
 \newcommand{\CN}[0]{\mathcal{N}}
\newcommand{\CO}[0]{\mathcal{O}} 
 \newcommand{\CR}[0]{\mathcal{R}}
 \newcommand{\CT}[0]{\mathcal{T}}
 
 \newcommand{\CX}[0]{\mathcal{X}}

\newcommand{\Span}[0]{\operatorname{span}}
\newcommand{\id}[0]{\operatorname{id}}		

\newcommand{\Aut}[0]{\operatorname{Aut}}

\newcommand*\onto{\ensuremath{\joinrel\relbar\joinrel\twoheadrightarrow}} 
\newcommand*\into{\ensuremath{\lhook\joinrel\relbar\joinrel\rightarrow}}  

\newtheorem{theorem}{Theorem}[section]
\newtheorem{lemma}[theorem]{Lemma}
\newtheorem{proposition}[theorem]{Proposition}
\newtheorem{corollary}[theorem]{Corollary}
\theoremstyle{remark}
\newtheorem{remark}[theorem]{\bfseries Remark}

\newtheorem{definition}[theorem]{\bfseries Definition}
\newtheorem{example}[theorem]{\bfseries Example}
\newtheorem{examples}[theorem]{\bfseries Examples}

\newtheorem*{theorem*}{Theorem}
\newtheorem*{lemma*}{Lemma}
\newtheorem*{remark*}{Remark}

\numberwithin{equation}{section}

\newcommand{\Z}{\mathbb{Z}}
\newcommand{\N}{\mathbb{N}}
\newcommand{\R}{\mathbb{R}}
\newcommand{\T}{\mathbb{T}}
\newcommand{\im}{\operatorname{im}}
\newcommand{\coker}{\operatorname{coker}}

\makeatletter
\ifcsname phantomsection\endcsname
    \newcommand*{\qrr@gobblenexttocentry}[5]{}
\else
    \newcommand*{\qrr@gobblenexttocentry}[4]{}
\fi
\newcommand*{\addsubsection}{%
    \addtocontents{toc}{\protect\qrr@gobblenexttocentry}%
    \subsection}
\makeatother

\begin{document}
\title{On C*-algebras of irreversible algebraic dynamical systems}

\author{Nicolai Stammeier}
\address{Mathematisches Institut, Westf\"{a}lischen Wilhelms-Universit\"{a}t M\"{u}nster \\ \newline\hspace*{4mm}Einsteinstrasse 62 \\ 48149 M\"{u}nster \\ Germany}
\email{n.stammeier@wwu.de}

\thanks{Research supported by DFG through SFB $878$ and by ERC through AdG $267079$.}

\date{\today}



\begin{abstract}
Extending the work of Cuntz and Vershik, we develop a general notion of independence for commuting group endomorphisms. Based on this concept, we initiate the study of irreversible algebraic dynamical systems, which can be thought of as irreversible analogues of the dynamical systems considered by Schmidt. To each irreversible algebraic dynamical system, we associate a universal C*-algebra and show that it is a UCT Kirchberg algebra under natural assumptions. Moreover, we discuss the structure of the core subalgebra, which turns out to be closely related to generalised Bunce-Deddens algebras in the sense of Orfanos. We also construct discrete product systems of Hilbert bimodules for irreversible algebraic dynamical systems which allow us to view the associated C*-algebras as Cuntz-Nica-Pimsner algebras. Besides, we prove a decomposition theorem for semigroup crossed products of unital C*-algebras by semidirect products of discrete, left cancellative monoids.
\end{abstract}

\maketitle
\section*{Introduction}\label{intro}
\noindent Let $G$ be a countable discrete group and $(\xi_g)_{g \in G}$ denote the standard orthonormal basis of the Hilbert space $\ell^2(G)$. Suppose $\varphi$ is an injective group endomorphism of $G$. Then $S_\varphi \xi_g = \xi_{\varphi(g)}$ defines an isometry on $\ell^2(G)$. For $g \in G$, let $U_g$ denote the canonical unitary on $\ell^2(G)$ given by left translation. Then $S_\varphi U_g = U_{\varphi(g)}S_\varphi$ holds for all $g \in G$. This leads to the C*-algebra $\CO_r[\varphi]$ generated by the isometry $S_\varphi$ and the unitaries $(U_g)_{g \in G}$. A natural object to study within this context is a universal model for $\CO_r[\varphi]$, which is a C*-algebra $\CO[\varphi] = C^*(\{s_\varphi,(u_g)_{g \in G} \mid \CR\})$ generated by an isometry $s_\varphi$ and unitaries $u_g$ satisfying a suitable set of relations $\CR$.

Suppose $\varphi$ is a group automorphism of $G$ which generates an effective $\Z$-action on $G$. Then $C^*(S_\varphi,(U_g)_{g \in G})$ is the crossed product $C^*_{r}(G) \rtimes_\alpha \Z$, where $\alpha(u_g) = u_{\varphi(g)}$. It is well-known that this crossed product is canonically isomorphic to the reduced group C*-algebra of the semidirect product $G \rtimes_\varphi \Z$. Hence, the universal model for $\CO_r[\varphi]$ is given by the full group C*-algebra of $G \rtimes_\varphi \Z$, provided that $G$ is amenable. The structure of these C*-algebras has already been studied extensively, see \cite{Wil}. In stark contrast, the situation for an injective, but non-surjective group endomorphism $\varphi$ has started to receive more attention in the recent past. The most elementary examples of such endomorphisms are $\times 2:\Z \longrightarrow \Z$ and the one-sided shift on $\bigoplus_{k \in \N} \Z/n\Z$ for $n \geq 2$. 

Restricting to the case where $G$ is amenable and $G/\varphi(G)$ is finite, Ilan Hirshberg introduced a universal C*-algebraic model $\CO[\varphi]$ for $\CO_r[\varphi]$ in 2002, see \cite{Hir}. He showed that the core $\CF \subset \CO[\varphi]$, which is the fixed point algebra under the canonical gauge action, is simple if $(\varphi^n(G))_{n \in \N}$ separates the points in $G$, that is, $\bigcap_{n \in \N} \varphi^n(G) = \{1_G\}$. Using simplicity of $\CF$, he concluded that $\CF$ is the crossed product of a natural commutative subalgebra $\CD$, called the diagonal, by $G$. Assuming that the family of subgroups $(\varphi^n(G))_{n \in \N}$ separates the points in $G$ and consists of normal subgroups of $G$, Hirshberg established that $\CO[\varphi]$ is simple and therefore isomorphic to $\CO_r[\varphi]$. Additionally, he computed the K-theory of $\CO[\varphi]$ based on the K-theory of $\CF$ and the Pimsner-Voiculescu six-term exact sequence for $\times n:\Z \longrightarrow \Z,n \geq 2$, the shift on $\bigoplus_\N H$, where $H$ is a finite group, and $\varphi:\Z/2\Z * \Z/2\Z \longrightarrow \Z/2\Z * \Z/2\Z, a \mapsto bab,b \mapsto aba$, where $a,b$ denote the standard generators of $\Z/2\Z * \Z/2\Z$.

A decade later, Felipe Vieira extended Hirshberg's results to the case where $G$ is amenable and $(\varphi^n(G))_{n \in \N}$ separates the points in $G$, see \cite{Vie}. His approach used techniques for semigroup crossed products as well as partial group crossed products. One remarkable outcome of his work is the connection to semigroup C*-algebras for left cancellative semigroups as introduced by Xin Li in \cites{Li1,Li2}: If $G$ is amenable, $(\varphi^n(G))_{n \in \N}$ separates the points in $G$, and $G/\varphi(G)$ is infinite, then $\CO[\varphi]$ is canonically isomorphic to the full semigroup C*-algebra of $G \rtimes_\varphi \N$. Furthermore, Vieira showed that this is the same as the reduced semigroup C*-algebra of $G \rtimes_\varphi \N$.

At about the same time, Joachim Cuntz and Anatoly Vershik examined the case where $G$ is abelian, $G/\varphi(G)$ is finite, and $(\varphi^n(G))_{n \in \N}$ separates the points in $G$, see \cite{CV}. They proved that $\CO[\varphi]$ is a UCT Kirchberg algebra and provided a general method to compute the K-theory of $\CO[\varphi]$. In addition, they found that the spectrum of the diagonal $\CD$ is a compact abelian group $G_\varphi$, which can be interpreted as a completion of $G$ with respect to $\varphi$. Another interesting outcome of \cite{CV} is the fact that $\CF \cong C(G_\varphi) \rtimes G$ is also isomorphic to $C(\hat{G}) \rtimes \hat{G}_\varphi$.

Summarizing the current status, it is fair to say that a lot is known about the C*-algebras $\CO[\varphi]$, $\CF$ and $\CD$ associated to a single injective, non-surjective group endomorphism $\varphi$ of a countably infinite, discrete group $G$. Indeed, in many cases we are able, at least in principle, to compute the K-theory for $\CO[\varphi]$, which is known to be a complete invariant due to the celebrated classification theorem by Eberhard Kirchberg and Christopher N. Phillips, see \cites{Kir,Phi}. Thus, by computing the K-theory of $\CO[\varphi]$, we can recover the information on the dynamical system $(G,\varphi)$ that is encoded in $\CO[\varphi]$. It is therefore natural to ask whether analogous results hold for similar dynamical systems involving more than one transformation. 

To motivate this question, let us mention an important example which showcases some interesting phenomena for such dynamical systems. In 1967, Hillel Furstenberg proved the following result, which applies for instance to $\times 2,\times 3:\IT \longrightarrow \IT$, the Pontryagin dual of $\times 2,\times 3:\Z \longrightarrow \Z$, see \cite{Fur}*{Part IV}: Every closed subset of $\IT$, which is invariant under the action of a non-lacunary subsemigroup of $\Z^\times$, is either finite or equals $\IT$. This led him to conjecture that a stronger form of rigidity might be true: Any invariant ergodic Borel probability measure on $\IT$ is either atomic or the Lebesgue measure on $\IT$. In its general form, this conjecture is still open. An important reduction step has been achieved by Daniel J. Rudolph, see \cite{Rud} and also \cite{Par} for a concise presentation. The conjecture has been verified by Manfred Einsiedler and Alexander Fish in 2010 for the case where the acting semigroup is sufficiently large in the sense that it has positive lower logarithmic density, see \cite{EF}. This form of measure rigidity has also been studied for certain reversible dynamical systems, see \cite{EK} and the references therein. In a different direction, Daniel J. Berend and Roman Muchnik generalised the rigidity result from \cite{Fur} stated above to compact abelian groups, see \cites{Ber1,Ber2,Muc}.

Coming back to $\times p,\times q:\IT \longrightarrow \IT$ for relatively prime integers $p,q \geq 2$, it is natural to ask: What are the essential features of this dynamical systems? By Pontryagin duality, it corresponds to $\times p,\times q:\Z \longrightarrow \Z$. The condition that $p$ and $q$ are relatively prime is mirrored both by $p\Z + q\Z = \Z$ and $p\Z \cap q\Z = pq\Z$. These simple facts led Joachim Cuntz and Anatoly Vershik to define the notion of independence for pairs of commuting injective group endomorphisms $\varphi$ and $\psi$ of a discrete abelian group $G$   with the restriction that $G/\varphi(G)$ and $G/\psi(G)$ be finite, see  \cite{CV}*{Section 5}: $\varphi$ and $\psi$ are said to be \emph{independent} if $\varphi(G) \cap \psi(G) = \varphi\psi(G)$. It is shown in \cite{CV}*{Lemma 5.1} that independence is equivalent to $\varphi(G)+\psi(G) = G$ as well as to the statement that the inclusion $\varphi(G) \hookrightarrow G$ induces an isomorphism $\varphi(G)/(\varphi(G) \cap \psi(G)) \cong G/\psi(G)$.  

In this paper, we will extend the notion of independence to the general case of two commuting injective group endomorphisms $\varphi$ and $\psi$ of a discrete group $G$. In particular, we show that the last equivalence still holds if we only ask for a bijection $\varphi(G)/(\varphi(G) \cap \psi(G)) \longrightarrow G/\psi(G)$, see Proposition~\ref{prop:IAD ind cond}. But $\varphi(G) \cap \psi(G) = \varphi\psi(G)$ turns out to be weaker than $\varphi(G)\psi(G) = G$, where $\varphi(G)\psi(G) = \{\varphi(g)\psi(g') \mid g,g' \in G\}$, see Example~\ref{ex:ind doesn't give strong ind}. We will therefore differentiate between independence and what we call strong independence, see Definition~\ref{def:ind}. An equivalent characterisation of independence can be given in terms of the isometries $S_\varphi, S_\psi \in \ell^2(G)$: The commuting endomorphisms $\varphi$ and $\psi$ are independent if and only if $S_\varphi^*S_\psi = S_\psi S_\varphi^*$ holds.

With this notion of independence for commuting injective group endomorphisms of discrete groups at our disposal, we can think of $\times 2,\times 3:\Z \longrightarrow \Z$ in an abstract way as a dynamical system $(G,P,\theta)$ given by
\begin{enumerate}[(A)]
\item a countably infinite, discrete group $G$ with unit $1_G$,
\item a countably generated, free abelian monoid $P$ with unit $1_P$, and 
\item a $P$-action $\theta$ on $G$ by injective group endomorphisms for which $\theta_p$ and $\theta_q$ are independent if and only if $p$ and $q$ are relatively prime.
\end{enumerate}
We will refer to triples $(G,P,\theta)$ satisfying the three requirements stated above as \emph{irreversible algebraic dynamical systems}. The term \emph{irreversible} is chosen because $\theta_p \in \Aut(G)$ implies $p = 1_P$, and \emph{algebraic} emphasizes the contrast to topological dynamical systems, since the imposed conditions are purely algebraic. More specifically, such dynamical systems can be regarded as irreversible analogues of algebraic dynamical systems as introduced by Klaus Schmidt, see \cites{Sch1,Sch2,LS} and the references therein. 

Within Section~2, we specialise to the case where the group $G$ is abelian. Using Pontryagin duality and some fact about annihilators, we arrive at a notion of independence for commuting surjective group endomorphisms of an arbitrary group, see Definition~\ref{def:ind surj} and also \cite{CV}*{Lemma 5.4}. This allows us to describe commutative irreversible algebraic dynamical systems $(G,P,\theta)$ entirely in terms of their dual models $(\hat{G},P,\hat{\theta})$, see Proposition~\ref{prop:CIAD as a topological dynamical system}. It is precisely this description which represents the close connection irreversible algebraic dynamical systems and irreversible $*$-commuting dynamical systems, see \cite{Sta2} for more information. 

Section~3 is devoted to the construction and study of a universal C*-algebra $\CO[G,P,\theta]$ associated to each irreversible algebraic dynamical system $(G,P,\theta)$. This C*-algebra is a direct generalisation of the C*-algebra $\CO[\varphi]$ that appeared in \cites{CV,Hir,Vie} and we show that the structural properties of $\CO[G,P,\theta]$ are in good accordance with the ones that have been found for $\CO[\varphi]$. More precisely, we prove that the spectrum $G_\theta$ of the (commutative) diagonal subalgebra $\CD$ of $\CO[G,P,\theta]$ can be interpreted as a completion of $G$ with respect to $\theta$ if $(G,P,\theta)$ is \emph{minimal} in the sense that $\bigcap_{p \in P}\theta_p(G) = \{1_G\}$, see Lemma~\ref{lem:IAD spec D as a completion of G}. This is a direct extension of \cite{CV}*{Lemma 2.4}. The C*-algebra $\CO[G,P,\theta]$ is then identified with the semigroup crossed product $\CD \rtimes (G \rtimes_\theta P)$, where $(g,p).d = u_gs_pd(u_gs_p)^*$, see Proposition~\ref{prop:O cong D rtimes GxP}. Using the decomposition theorem for crossed products by semidirect products of monoids provided in the appendix, see Theorem~\ref{thm:cr prod by sd prod as it cr prod}, the isomorphism between $\CO[G,P,\theta]$ and $\CD \rtimes (G\rtimes_\theta P)$ yields an isomorphism of $\CF$ and $C(G_\theta) \rtimes_\tau G$, where $g.d = u_gdu_g^*$, see Corollary~\ref{cor:F cong D rtimes G-general IAD}. 

As a next step, we show that minimality of $(G,P,\theta)$ and amenability of the $G$-action $\hat{\tau}$ on $G_\theta$ are sufficient for simplicity and pure infiniteness of $\CO[G,P,\theta]$, see Theorem~\ref{thm:O p.i. and simple}. The general idea of the proof of this result goes back to \cite{Cun}, but the technical details are more involved compared to the singly generated case, see \cite{CV}*{Theorem 2.6}. But with this result at hands, we get that minimality of $(G,P,\theta)$ and amenability of $\hat{\tau}$ imply that $\CO[G,P,\theta]$ is a UCT Kirchberg algebra, hence classifiable by K-theory due to \cites{Kir,Phi}, see Corollary~\ref{cor:UCT Kirchberg algebra}. Unfortunately, the computation of the K-theory of $\CO[G,P,\theta]$, beyond the case of a single group endomorphism for which this has been accomplished in \cite{CV}, is a hard problem, at least with the techniques currently available.

In Section~4, we restrict our focus to the case where $G/\theta_p(G)$ is finite for all $p \in P$. We find that, in case $G$ is amenable and $(G,P,\theta)$ is minimal, the core $\CF$ is a generalised Bunce-Deddens algebra in the sense of \cite{Orf}, see Proposition~\ref{prop:gen BD-alg for IADoFT} and \cite{Orf}. In this case, $\CF$ is classified by its Elliott invariant due to a combination of results from \cites{Lin,MS,Win}, see Corollary~\ref{cor:IAD class of the core}. In addition, we find an intriguing chain of isomorphisms $\CF \cong C(G_\theta) \rtimes_\tau G \cong C(\hat{G}) \rtimes_{\bar{\tau}} \hat{G}_\theta$ in the case where $(G,P,\theta)$ is minimal and $G$ is commutative, see Corollary~\ref{cor:dual cr prod for aIADoFT}. The corresponding result for the case of a single group endomorphism was established in \cite{CV}*{Section 2}.

Section~5 provides an alternative approach to the C*-algebra $\CO[G,P,\theta]$ as the Cuntz-Nica-Pimsner algebra of a discrete product systems of Hilbert bimodules naturally associated to $(G,P,\theta)$, see Theorem~\ref{thm:isom ad-hoc PS for IADoFT}. Discrete product systems form a generalisation of the original construction introduced by Mihai Pimsner in \cite{Pim} for a single Hilbert bimodule. We refer to \cites{Fow1,Fow2,Sol,Yee,SY,CLSV,HLS} for more information on the subject. One interesting aspect is that the product system $\CX$ associated to $(G,P,\theta)$ comes with a canonical system of orthonormal bases on its fibres $\CX_p$, obtained by choosing a transversal for $G/\theta_p(G)$, see Proposition~\ref{prop:PS for an IAD}. 

A particular advantage of realizing $\CO[G,P,\theta]$ as the Cuntz-Nica-Pimsner algebra of the product system $\CX$ is that it has a natural Toeplitz extension, called the Nica-Toeplitz algebra. 
This algebra will be studied in a forthcoming paper together with Nathan Brownlowe and Nadia S.Larsen, where we show that the Nica-Toeplitz algebra associated to an irreversible algebraic dynamical system $(G,P,\theta)$ is canonically isomorphic to the (full) semigroup C*-algebra $C^*(G \rtimes_\theta P)$ in the sense of Xin Li, see \cites{Li1,Li2}. In fact, we will prove this in a more general context where $P$ may be an arbitrary right LCM semigroup in the sense of \cite{BLS1}. Moreover, this C*-algebra coincides with $\CO[G,P,\theta]$ for irreversible algebraic dynamical systems of infinite type $(G,P,\theta)$, that is, $G/\theta_p(G)$ is infinite for all $p \neq 1_P$. This sheds new light on the results from \cite{Vie} mentioned in the beginning.\vspace*{3mm}

\noindent \textit{Acknowledgements.} This paper constitutes one part of the author's doctoral thesis completed at Westf\"alische Wilhelms-Universit\"at M\"unster under the supervision of Joachim Cuntz, whose support is highly appreciated. The author is also grateful to Nadia S. Larsen and Nathan Brownlowe for numerous discussions on related topics, to Jos\'{e} Carri\'{o}n for discussing generalised Bunce-Deddens algebras, to Wilhelm Winter for sharing the idea behind Corollary~\ref{cor:IAD class of the core}, as well as to Sel\c{c}uk Barlak and G\'{a}bor Szab\'{o} for valuable feedback on preliminary versions.

\section{Irreversible algebraic dynamical systems}\label{sec1}	
\noindent The purpose of this section is to familiarize with the primary object of interest called irreversible algebraic dynamical system in its most general form. Vaguely speaking, such a dynamical system is given by a countably infinite, discrete group $G$ and at most countably many commuting injective, non-surjective group endomorphisms $(\theta_i)_{i \in I}$ of $G$ that are independent in the sense that the intersection of their images is as small as possible. Additionally, we will introduce a minimality condition stating that the intersection of the images of the group endomorphisms from the semigroup generated by $(\theta_i)_{i \in I}$ is trivial. In other words, the group endomorphisms $(\theta_i)_{i \in I}$ (more precisely, finite products of these) separate the points in $G$. At a later stage, namely in Theorem~\ref{thm:O p.i. and simple}, this condition is shown to be intimately connected to simplicity of the C*-algebra $\CO[G,P,\theta]$ associated to such a dynamical system in Definition~\ref{def:O-algebra ad-hoc}.

The following observation is an extension of the concept of independence introduced in \cite{CV}*{Section 5}. In contrast to the situation in \cite{CV}, we will require neither the group $G$ to be abelian nor the cokernels of the injective group endomorphisms of $G$ to be finite.

\begin{proposition}\label{prop:IAD ind cond}\index{independence conditions}
Suppose $G$ is a group. Consider the following statements for two commuting injective group endomorphisms $\theta_{1}$ and $\theta_2$ of $G$:
\begin{enumerate}[(i)]
\item $\theta_{1}(G) \theta_{2}(G) = G$.
\item The map $\theta_{1}(G)/(\theta_{1}(G) \cap \theta_{2}(G)) \longrightarrow G/\theta_{2}(G)$ induced by the inclusion $\theta_{1}(G) \hookrightarrow G$ is a bijection.
\item[(ii')] The map $\theta_{2}(G)/(\theta_{1}(G) \cap \theta_{2}(G)) \longrightarrow G/\theta_{1}(G)$ induced by the inclusion $\theta_{2}(G) \hookrightarrow G$ is a bijection.
\item $\theta_1(G) \cap \theta_2(G) = \theta_1\theta_2(G)$.
\end{enumerate}
Then (i),(ii), and (ii') are equivalent and imply (iii). If either of the subgroups $\theta_1(G)$ or $\theta_2(G)$ is of finite index in $G$, then (i)--(iii) are equivalent.
\end{proposition}
\begin{proof}
Note that we always have $\theta_1(G)\theta_2(G) \subset G$ and $\theta_1(G) \cap \theta_2(G) \supset \theta_1\theta_2(G)$. Moreover, in condition (ii), the inclusion $\theta_1(G) \hookrightarrow G$ induces an injective map $\theta_1(G)/(\theta_1(G) \cap \theta_2(G)) \longrightarrow G/\theta_2(G)$. The corresponding statement holds for (ii').\\ 
If (i) holds true, then $G \ni g = \theta_1(g_1)\theta_2(g_2)$ for suitable $g_i \in G$. Hence, the left-coset of $\theta_1(g_1)$ maps to the left-coset of $g$ and (ii) follows.

Conversely, suppose (ii) is valid and pick $g \in G$. Then there is $g_1 \in G$ such that $\theta_1(g_1)\left(\theta_1(G) \cap \theta_2(G)\right) \mapsto g\theta_2(G)$ via the map from (ii). But since this map comes from the inclusion $\theta_1(G) \into G$, we have $g\theta_2(G) = \theta_1(g_1)\theta_2(G)$. Thus, there is $g_2 \in G$ such that $g = \theta_1(g_1)\theta_2(g_2)$ showing (i). The equivalence of (i) and (ii') is obtained from the previous argument by swapping $\theta_1$ and $\theta_2$. Given (ii), that is, 
\[f_1: \theta_1(G)/(\theta_1(G) \cap \theta_2(G)) \longrightarrow G/\theta_2(G)\] 
is a bijection (induced by $\theta_2(G) \hookrightarrow G$), composing $f_1^{-1}$ with the bijection
\[f_2: \theta_1(G)/(\theta_1\theta_2(G)) \longrightarrow G/\theta_2(G)\]
obtained from injectivity of $\theta_1$ yields a bijection
\[f_1^{-1}f_2: \theta_1(G)/(\theta_1\theta_2(G)) \longrightarrow \theta_1(G)/(\theta_1(G) \cap \theta_2(G)).\]
Let us assume $\theta_1\theta_2(G) \subsetneqq \theta_1(G) \cap \theta_2(G)$. This means, that there is $g \in \theta_1(G)$ such that $g \theta_1\theta_2(G) \neq \theta_1\theta_2(G)$ but $g\theta_1(G) \cap \theta_2(G) = \theta_1(G) \cap \theta_2(G)$. Noting that $f_1^{-1}f_2$ maps a left-coset $g'\theta_1\theta_2(G)$ to $g'\theta_1(G) \cap \theta_2(G)$, this contradicts injectivity of $f_1^{-1}f_2$. Hence, we must have $\theta_1(G) \cap \theta_2(G) = \theta_1\theta_2(G)$. Similarly, (iii) follows from (ii').\\  
Finally, suppose (iii) holds. By injectivity of $\theta_1$, we have 
\[\theta_1(G)/(\theta_1(G) \cap \theta_2(G)) = \theta_1(G)/\theta_1\theta_2(G) \cong G/\theta_2(G).\]
So if $[G : \theta_2(G)]$ is finite, then the injective map from (ii) is necessarily a bijection. If $[G : \theta_1(G)]$ is finite, we get (ii') in the same manner.  
\end{proof}

\begin{remark}\label{rem:ind normal}
If the subgroups $\theta_1(G)$ and $\theta_2(G)$ are both normal in $G$, then $\theta_1(G)\theta_2(G)$ is a normal subgroup of $\theta_i(G), i=1,2$, and the bijections in Proposition~\ref{prop:IAD ind cond}~(ii) and (ii') are isomorphisms of groups.
\end{remark}

\begin{definition}\label{def:ind}
Let $G$ be a group and $\theta_{1},\theta_2$ commuting, injective group endomorphisms of $G$. Then $\theta_1$ and $\theta_2$ are said to be \emph{independent}, if they satisfy condition (iii) from Proposition~\ref{prop:IAD ind cond}. $\theta_1$ and $\theta_2$ are said to be \emph{strongly independent}, if they satisfy the condition (i) from Proposition~\ref{prop:IAD ind cond}. 
\end{definition}

\noindent Note that (strong) independence holds if $\theta_1$ or $\theta_2$ is an automorphism.

\begin{lemma}\label{lem:ind under prod for inj endos}
Let $G$ be a group and suppose $\theta_{1},\theta_2,\theta_3$ are commuting, injective group endomorphisms of $G$. $\theta_1$ is (strongly) independent of $\theta_2\theta_2$ if and only if $\theta_1$ is (strongly) independent of both $\theta_2$ and $\theta_3$.
\end{lemma}
\begin{proof}
If $\theta_1$ and $\theta_2\theta_3$ are strongly independent, then 
\[\theta_1(G)\theta_2(G) \supset \theta_1(G)\theta_2(\theta_3(G)) = G\]
shows that $\theta_1$ and $\theta_2$ are strongly independent. As $\theta_2$ and $\theta_3$ commute, $\theta_1$ is also strongly independent of $\theta_3$. Conversely, if $\theta_1$ is strongly independent of both $\theta_2$ and $\theta_3$, then
\[\begin{array}{lclcl}
G &=& \theta_1(G)\theta_2(G) &=& \theta_1(G)\theta_2(\theta_1(G)\theta_3(G))\vspace*{2mm}\\
&=& \theta_1(G\theta_2(G))\theta_2(\theta_3(G)) &\subset& \theta_1(G)\theta_2\theta_3(G),
\end{array}\]
so $\theta_1$ and $\theta_2\theta_3$ are strongly independent since the reverse inclusion is trivial. If $\theta_1$ and $\theta_2\theta_3$ are independent, then commutativity of $\theta_1,\theta_2$ and $\theta_3$ in combination with injectivity of $\theta_3$ yield
\[\begin{array}{lclcl}
\theta_1(G) \cap \theta_2(G) &=& \theta_3^{-1}(\theta_1\theta_3(G) \cap \theta_2\theta_3(G)) &\subset&  \theta_3^{-1}(\theta_1(G) \cap \theta_2\theta_3(G))\vspace*{2mm}\\ 
&=&  \theta_3^{-1}(\theta_1\theta_2\theta_3(G)) &=& \theta_1\theta_2(G).
\end{array}\]
Since the reverse inclusion is always true, we conclude that $\theta_1$ and $\theta_2$ are independent. Exchanging the role of $\theta_2$ and $\theta_3$ shows independence of $\theta_1$ and $\theta_3$.
Finally, if $\theta_1$ is independent of both $\theta_2$ and $\theta_3$, we get
\[\begin{array}{lclcl}
\theta_1(G) \cap \theta_2\theta_3(G) &=& \theta_1(G) \cap \theta_2(G) \cap \theta_2\theta_3(G) &=& \theta_1\theta_2(G) \cap \theta_2\theta_3(G)\vspace*{2mm}\\ 
&=& \theta_2(\theta_1(G) \cap \theta_3(G)) &=& \theta_1\theta_2\theta_3(G)
\end{array}\]
by injectivity of $\theta_2$. Thus $\theta_1$ and $\theta_2\theta_3$ are independent.
\end{proof}

\noindent If $(P,\leq)$ is a lattice-ordered monoid with unit $1_P$, we shall denote the least common multiple and the greatest common divisor of two elements $p,q \in P$ by $p \vee q$ and $p \wedge q$, respectively. $p$ and $q$ are said to be relatively prime (in $P$) if $p \wedge q = 1_P$ or, equivalently, $p \vee q = pq$. Simple examples of such monoids are countably generated free abelian monoids since such monoids are either isomorphic to $\N^k$ for some $k \in \N$ or $\bigoplus_\N \N$.

\begin{definition}\label{def:IAD}\label{def:aIAD and aIADoFT} 
An \emph{irreversible algebraic dynamical system} $(G,P,\theta)$ is 
\begin{enumerate}[(A)]
\item a countably infinite, discrete group $G$ with unit $1_G$,
\item a countably generated, free abelian monoid $P$ with unit $1_P$, and 
\item a $P$-action $\theta$ on $G$ by injective group endomorphisms for which $\theta_p$ and $\theta_q$ are independent if and only if $p$ and $q$ are relatively prime. 
\end{enumerate}
An irreversible algebraic dynamical system $(G,P,\theta)$ is said to be
\begin{enumerate}[$\cdot$]
\item \emph{minimal}, if $\bigcap_{p \in P}{\theta_{p}(G)} = \{1_G\}$,
\item \emph{commutative}, if $G$ is commutative, 
\item \emph{of finite type}, if $[G : \theta_p(G)]$ is finite for all $p \in P$, and
\item \emph{of infinite type}, if $[G : \theta_p(G)]$ is infinite for all $p \neq 1_P$.
\end{enumerate} 
\end{definition}

\begin{remark}\label{rem:IAD first remarks_change}\label{rem:IAD ind cond} 
$\theta_{1_P} = id_G$ is the only automorphism of $G$ occurring for this setting. Indeed, if $\theta_p$ is an automorphism of $G$, it is independent of itself. So unless $P = \{1_P\}$, there is $p \in P$ such that $[G : \theta_{p}(G)] > 1$. Therefore, $\theta_p(G)$ is a proper subgroup of $G$. Since $\theta_p$ is injective, $G$ has to be infinite.
\end{remark}   

\begin{remark}
The minimality condition has been used under the name \emph{exactness} in the case of a commutative $G$ with a single endomorphism with finite cokernel in \cite{CV}. As explained in \cite{CV}*{Remark 2.1}, the notion of exactness for a single endomorphism stems from ergodic theory and is a well-studied property for irreversible, measure-preserving transformations. However, for the specific setup that we use, this property was already considered by Ilan Hirshberg in \cite{Hir}, where he called such endomorphisms \emph{pure}. Nevertheless, we refer to this property as minimality for two reasons: 
\begin{enumerate}[1.]
\item For commutative irreversible algebraic dynamical systems, the corresponding condition for the dual model $(\hat{G},P,\hat{\theta})$ is minimality of the (irreversible) topological dynamical system, see Proposition~\ref{prop:CIAD as a topological dynamical system}. 
\item The property is intimately linked to simplicity of the C*-algebras $\CO[G,P,\theta]$ and $\CF$, see Corollary~\ref{cor:D rtimes G simple iff IAD minimal} and Theorem~\ref{thm:O p.i. and simple}.
\end{enumerate}
\end{remark}

\begin{examples}\label{ex:IAD standard}
There are various examples for commutative irreversible algebraic dynamical systems and most of them are of finite type. Let us recall that it suffices to check independence of the endomorphisms on the generators of $P$ according to Lemma~\ref{lem:ind under prod for inj endos}.
\begin{enumerate}[(a)]
\item Choose a family $(p_i)_{i \in I} \subset \Z^\times{\setminus}\Z^* = \Z{\setminus}\{0,\pm 1\}$ and let $P = |(p_i)_{i \in I}\rangle$ act on $G=\Z$ by $\theta_{p_i}(g) = p_ig$. Since $\Z$ is an integral domain, each $\theta_{p_i}$ is an injective group endomorphism of $G$ with $[G:\theta_{p_i}(G)] = p_i$. For $i \neq j$, $\theta_{p_i}$ and $\theta_{p_j}$ are independent if and only if $p_i$ and $p_j$ are relatively prime in $\Z$. Thus, we get a commutative irreversible algebraic dynamical system of finite type if and only if $(p_i)_{i \in I}$ consists of relatively prime integers. Since the number of factors in its prime factorization is finite for every integer, such irreversible algebraic dynamical systems are automatically minimal. 
\item Let $I\subset\N$, choose relatively prime integers $\{q\} \cup (p_i)_{i \in I} \subset \Z^\times{\setminus}\Z^*$ and let $G = \Z[1/q]$. As $\Z[1/q] = \varinjlim \Z$ with connecting maps given by multiplication with $q$, and $q$ is relatively prime to each $p_i$, the arguments from (a) carry over almost verbatim. Thus we get minimal commutative irreversible algebraic dynamical systems of finite type $(G,P,\theta)$ which generalise \cite{CV}*{Example 2.1.5}.
\item Let $\IK$ be a countable field and let $G=\IK[T]$ denote the polynomial ring in a single variable $T$ over $\IK$. Choose non-constant polynomials $p_i \in \IK[T], i \in I$. Multiplying by $p_i$ defines an endomorphism $\theta_{p_i}$ of $G$ with $[G:\theta_{p_i}(G)] = |\IK|^{\deg(p_i)}$, where $\deg(p_i)$ denotes the degree of $p_i \in \IK[T]$. Thus, if we let $P := |(p_i)_{i \in I}\rangle$, then the index of $\theta_p(G)$ in $G$ is finite for all $p \in P$ if and only if $\IK$ is finite. It is clear that $\theta_{p_i}$ and $\theta_{p_j}$ are independent if and only if $(p_i) \cap (p_j) = (p_ip_j)$ holds for the principal ideals (whenever $i \neq j$). Since every $g \in \IK[T]$ has finite degree, $(G,P,\theta)$ is automatically minimal. Thus, provided $(p_i)_{i \in I}$ has been chosen accordingly, we obtain a minimal commutative irreversible algebraic dynamical system which is of finite type if and only if $\IK$ is finite, compare \cite{CV}*{Example 2.1.4}.  
\end{enumerate}
\end{examples}

\begin{example}\label{ex:IAD integer matrices}
For $G=\Z^d$ with $d \geq 1$, the monoid of injective group endomorphisms of $G$ is isomorphic to the monoid of invertible integral matrices $M_d(\Z) \cap Gl_d(\IQ)$. For each such endomorphism, the index of its image in $G$ is given by the absolute value of the determinant of the corresponding matrix. In particular, their images always have finite index in $G$ and an endomorphism of $G$ is not surjective precisely if the absolute value of the determinant of the matrix exceeds $1$. So let $(T_i)_{i \in I} \subset M_d(\Z) \cap Gl_d(\IQ)$ be a family of commuting matrices satisfying $|\det T_i| > 1$ for all $i \in I$ and set $P = |(T_i)_{i \in I}\rangle$ as well as $\theta_i(g) = T_ig$. For $i \neq j$, it is easier to check strong independence of $\theta_i$ and $\theta_j$ instead of independence. Indeed, since we are dealing with a finite type case, the two conditions are equivalent and strong independence takes the form $T_i(\Z^d)+T_j(\Z^d) = \Z^d$, see Proposition~\ref{prop:IAD ind cond}. This condition can readily be checked. Moreover, minimality is related to generalised eigenvalues and we note that, in the case where $P$ is singly generated, the generating integer matrix has to be a \emph{dilation matrix}. This situation has been studied extensively in \cite{EaHR}.
\end{example}

\noindent Example~\ref{ex:IAD standard}~(a) can be generalised to the case of rings of integers:

\begin{example}\label{ex:IAD rings of integers}
Let $\CR$ be the ring of integers in a number field and denote by $\CR^\times = \CR\setminus\{0_\CR\}$ the multiplicative subsemigroup as well as by $\CR^* \subset \CR^\times$ the group of units in $\CR$. Take $G = \CR$ and choose a (countable) family $(p_i)_{i \in I} \subset \CR^\times\setminus\CR^*$. If we set $P = |(p_i)_{i \in I}\rangle$, then this monoid acts on $G$ by multiplication, i.e. $\theta_p(g) = pg$ for $g \in G, p \in P$. For $i \neq j$, $\theta_{p_i}$ and $\theta_{p_j}$ are independent if and only if the principal ideals $(p_i)$ and $(p_j)$ in $\CR$ have no common prime ideal. If this is the case, $(G,P,\theta)$ constitutes a commutative irreversible algebraic dynamical system of finite type. Since the number of factors in the (unique) prime ideal factorization of $(g)$ in $\CR$ is finite for every $g \in G$, minimality is once again automatically satisfied.
\end{example}

\noindent As a matter of fact, the construction from Example~\ref{ex:IAD rings of integers} is applicable to Dedekind domains $\CR$. Next, we would like to mention the following example even though, having singly generated $P$, it has nothing to do with independence. The reason is that Joachim Cuntz and Anatoly Vershik observed in \cite{CV}*{Example 2.1.1}, that the C*-algebra $\CO[G,P,\theta]$ associated to this irreversible algebraic dynamical system is isomorphic to $\CO_n$.

\begin{example}\label{ex:IAD for On}
For $n \geq 2$, consider the unilateral shift $\theta_1$ acting on $G = \bigoplus_\N \Z/n\Z$ by $(g_0,g_1,\dots) \mapsto (0,g_0,g_1,\dots)$. Since $\theta_1$ is an injective group endomorphism with $[G:\theta_1(G)] = n$, $(G,P,\theta)$ with $P = |\theta_1\rangle$ is a minimal commutative irreversible algebraic dynamical system of finite type.
\end{example}

\begin{example}\label{ex:ind doesn't give strong ind}
Generalising Example~\ref{ex:IAD for On}, suppose $P$ is as required in condition (B) of Definition~\ref{def:IAD} and let $G_0$ be a countable group. Let us assume that $G_0$ has at least two distinct elements. Then $P$ admits a shift action $\theta$ on $G := \bigoplus_P G_0$ given by $(\theta_p((g_q)_{q \in P}))_r = \chi_{pP}(r)~g_{p^{-1}r} \text{ for all } p,r \in P.$ It is apparent that $\theta_p\theta_q = \theta_q\theta_p$ holds for all $p,q \in P$ and that $\theta_p$ is an injective group endomorphism for all $p \in P$. The index $[G:\theta_p(G)]$ is finite for $p \in P\setminus\{1_P\}$ if and only if $G_0$ is finite and $P$ is singly generated. Indeed, if $p \neq 1_P$, then each element of $\bigoplus_{q \in P\setminus pP} G_0$ yields a distinct left-coset in $G/\theta_p(G)$. Clearly, this group is finite if and only if $G_0$ is finite and $P$ is singly generated. Given relatively prime $p$ and $q$ in $P\setminus\{1_P\}$, $\theta_p(G)\theta_q(G) \neq G$ since $g_{1_P} = 1_{G_0}$ for all $(g_r)_{r \in P} \in \theta_p(G)\theta_q(G)$ as $1_P \notin pP \cup qP$. Thus, unless $P$ is singly generated, $\theta$ does not satisfy the strong independence condition. However, the independence condition is satisfied because $g = (g_r)_{r \in P} \in \theta_p(G)\cap\theta_q(G)$ implies that $g_r \neq 1_{G_0}$ only if $r \in pP \cap qP = pqP$ and thus $g \in \theta_{pq}(G)$.
\end{example}

\noindent We have seen in Example~\ref{ex:ind doesn't give strong ind} that one cannot expect strong independence for irreversible algebraic dynamical systems of infinite type in general. On the other hand, there are some examples where the subgroups in question have infinite index and the endomorphisms are strongly independent:

\begin{example}\label{ex:IAD as sum of IADs}\label{ex:strong ind with inf index}
Given a family $(G^{(i)},P,\theta^{(i)})_{i\in \N}$ of irreversible algebraic dynamical systems, we can consider $G := \bigoplus_{i \in \N} G^{(i)}$. If $P$ acts on $G$ component-wise, i.e. $\theta_p (g_i)_{i \in \N} := (\theta^{(i)}_p(g_i))_{i \in \N}$, then $(G,P,\theta)$ is an irreversible algebraic dynamical system and $[G : \theta_p(G)]$ is infinite unless $p = 1_P$. $G$ is commutative if and only if each $G^{(i)}$ is, and $(G,P,\theta)$ is minimal if and only if each $(G^{(i)},P,\theta^{(i)})$ is minimal. If each $(G^{(i)},P,\theta^{(i)})$ satisfies the strong independence condition, then $\theta$ inherits this property as well.
\end{example}

\noindent As a final example, we provide more general forms of \cite{Vie}*{Example 2.3.9}. These examples are neither commutative irreversible algebraic dynamical systems nor of finite type.

\begin{example}\label{ex:IAD free group}
For $2 \leq n \leq \infty$, let $\IF_n$ be the free group in $n$ generators $(a_k)_{1 \leq k \leq n}$. Fix $1 \leq d \leq n$ and choose for each $1 \leq i \leq d$ an $n$-tuple $(m_{i,k})_{1 \leq k \leq n} \subset \N^\times$ such that 
\begin{enumerate}[1)]
\item there exists $k$ such that $m_{i,k} > 1$ for each $1 \leq i \leq d$, and
\item $m_{i,k}$ and $m_{j,k}$ are relatively prime for all $i \neq j, 1 \leq k \leq n$.
\end{enumerate}
Then $\theta_i(a_k) = a_k^{m_{i,k}}$ defines a group endomorphism of $\IF_n$ for each $1 \leq i \leq d$. Noting that the length of an element of $\IF_n$ in terms of the generators $(a_k)_{1 \leq k \leq n}$ and their inverses is non-decreasing under $\theta_i$, we deduce that $\theta_i$ is injective. It is clear that $\theta_i\theta_j = \theta_j\theta_i$ holds for all $i$ and $j$. For every $1 \leq i \leq d$, the index $[\IF_n:\theta_i(\IF_n)]$ is infinite. Indeed, take $1 \leq k \leq n$ such that $m_{i,k} > 1$ according to 1) and pick $1 \leq \ell \leq n$ with $\ell \neq k$. Then the family $((a_k a_\ell)^j)_{j \geq 1}$ yields pairwise distinct left-cosets in $\IF_n/\theta_i(\IF_n)$ since reduced words of the form $a_k a_\ell b\dots$ with $b \neq a_\ell^{-1}$ are not contained in $\theta_i(\IF_n)$. A similar argument shows that $\theta_i$ and $\theta_j$ are not strongly independent for $i \neq j$: By 1), there are $1 \leq k,\ell \leq n$ such that $m_{i,k} > 1$ and $m_{j,\ell} > 1$. This forces $a_k a_\ell \notin \theta_i(\IF_n)\theta_j(\IF_n)$. Nonetheless, $\theta_i$ and $\theta_j$ are independent due to 2). Thus, $G = \IF_n$ and $P = |(\theta_i)_{1 \leq i \leq d}\rangle$ acting on $G$ in the obvious way constitutes an irreversible algebraic dynamical system which is neither commutative nor of finite type. Minimality of such irreversible algebraic dynamical systems can easily be characterized by:
\begin{enumerate}
\item[3)] For each $1 \leq k \leq n$, there exists $1 \leq i \leq d$ satisfying $m_{i,k} > 1$.
\end{enumerate}
\end{example}

\noindent In addition to the presented spectrum of examples, we would like to mention that there are also examples of minimal, commutative irreversible algebraic dynamical systems of finite type arising from cellular automata, see \cite{Sta2}*{Example 1.19  and Example 1.21}. 

We close this section with two preparatory lemmas which are relevant for the C*-algebraic considerations in Section~3. The first lemma reflects a crucial feature of the independence assumption.

\begin{lemma}\label{lem:IAD intersection formula}
If $(G,P,\theta)$ is an irreversible algebraic dynamical system,
\[g\theta_p(G) \cap h\theta_q(G) = \begin{cases} g\theta_p(h')\theta_{p \vee q}(G) &\text{ if } g^{-1}h \in \theta_p(G)\theta_q(G),\\ \emptyset &\text{ else} \end{cases}\]
holds for all $g,h \in G, p,q \in P$, where $h'$ is uniquely determined by $g\theta_p(h') \in h\theta_q(G)$ up to multiplication from the right by elements from $\theta_{p^{-1}(p \vee q)}(G)$.
\end{lemma}
\begin{proof}
If there exist $g_1,g_2 \in G$ such that $g\theta_p(g_1) = h\theta_q(g_2)$, then $g^{-1}h = \theta_p(g_1)\theta_q(g_2^{-1}) \in \theta_p(G)\theta_q(G)$ follows because $G$ is group. Now suppose that $g_3,g_4 \in G$ satisfy $g\theta_p(g_3) = h\theta_q(g_4)$ as well. Since this implies $\theta_p(g_1^{-1}g_3) = \theta_q(g_2^{-1}g_4)$, we deduce $\theta_p(g_1^{-1}g_3) \in \theta_{p \vee q}(G)$. Using injectivity of $\theta_p$, this is equivalent to $g_1^{-1}g_3 \in \theta_{p^{-1}(p \vee q)}(G)$. Therefore, $h' = g_1$ is unique up to right multiplication by elements from $\theta_{p^{-1}(p \vee q)}(G)$.
\end{proof}

\noindent For the proof of Theorem~\ref{thm:O p.i. and simple}, we will need the following auxiliary result, which relies on irreversibility of the dynamical system:

\begin{lemma}\label{lem:believe it or not}
Suppose $(G,P,\theta)$ is an irreversible algebraic dynamical system and we have $n \in \N, g_{i} \in G, p_{i} \in P \setminus \{1_P\}$ for $0 \leq i \leq n$. Then, there exist $g \in g_{0}\theta_{p_{0}}(G), p \in p_{0}P$ satisfying 
\[\begin{array}{c} g\theta_{p}(G) \subset G \setminus \bigcup\limits_{1 \leq i \leq n}\left(g_{i}\bigcap\limits_{m \in \N}{\theta_{p_{i}^{m}}(G)}\right). \end{array}\]
\end{lemma}
\begin{proof}
We proceed by induction starting with $n = 1$. As $p_{1} \neq e$, we can find $m \in \N$ such that $p_{0} \notin p_{1}^{m}P$. Thus we have $p_{0} \vee p_{1}^{m} \gneqq p_{0}$. By Lemma~\ref{lem:IAD intersection formula}, 
\[\begin{array}{c} \hspace*{-1mm}(g_{0}\theta_{p_{0}}(G)) \cap (g_{1}\theta_{p_{1}^{m}}(G)) = \begin{cases} g_{0}\theta_{p_{0}}(\tilde{g}_{1})\theta_{p_{0} \vee p_{1}^{m}}(G) &\hspace*{-3.5mm} \text{ if } g_{0}^{-1}g_{1} \in \theta_{p_{0}}(G) \theta_{p_{1}^{m}}(G), \\ \emptyset &\hspace*{-3.5mm} \text{ else,} \end{cases} \end{array}\]
where $\tilde{g}_{1}$ is uniquely determined up to $\theta_{p_{0}^{-1}(p_{0} \vee p_{1}^{m})}(G)$. While $g := g_{0}$ works in the second case, we need $g \in (g_{0}\theta_{p_{0}}(G)) \setminus g_{0}\theta_{p_{0}}(\tilde{g}_{1})\theta_{p_{0} \vee p_{1}^{m}}(G)$ in the first case. Note that such a $g$ exists as $p_{0} \vee p_{1}^{m} \gneqq p_{0}$ by the choice of $m$ and we set $p := p_{0} \vee p_{1}^{m}$. 

The induction step from $n$ to $n+1$ is just a verbatim repetition of the first step: Assume that the statement holds for fixed $n$. This means that there exist $h \in g_0\theta_{p_0}(G)$ and $q \in p_0P$ such that 
\[\begin{array}{c} h\theta_{q}(G) \subset G \setminus \bigcup\limits_{1 \leq i \leq n}\left(g_{i}\bigcap\limits_{m \in \N}{\theta_{p_{i}^{m}}(G)}\right). \end{array}\] 
As $p_{n+1} \neq e$, we can find $m \in \N$ such that $q \notin p_{n+1}^{m}P$. In other words, we have $q \vee p_{n+1}^{m} \gneqq q$. Recall that 
\[\begin{array}{ll} \hspace*{-2mm}(h\theta_{q}(G)) \cap (g_{n+1}\theta_{p_{n+1}^{m}}(G)) \vspace*{2mm}\\ &\hspace*{-24mm}= \begin{cases} h\theta_{q}(\tilde{g}_{n+1})\theta_{q \vee p_{n+1}^{m}}(G) & \text{ if } h^{-1}g_{n+1} \notin \theta_q(G) \theta_{p_{n+1}^{m}}(G),\\ \emptyset & \text{ else,}\end{cases} \end{array}\]
where $\tilde{g}_{n+1}$ is uniquely determined up to $\theta_{q^{-1}(q \vee p_{n+1}^{m})}(G)$. In the second case, take $g := h$. For the first case, we choose $g \in (h\theta_{q}(G)) \setminus h\theta_{q}(\tilde{g}_{n+1})\theta_{q \vee p_{n+1}^{m}}(G)$. Note that such a $g$ exists as $q \vee p_{n+1}^{m} \gneqq q$ by the choice of $m$. Finally, let $p := q \vee p_{n+1}^{m}$. Then, it is clear from the construction that we indeed have 
\[\begin{array}{c} g\theta_{p}(G) \subset G \setminus \bigcup\limits_{1 \leq i \leq n+1}\left(g_{i}\bigcap\limits_{m \in \N}{\theta_{p_{i}^{m}}(G)}\right). \end{array}\]
\end{proof}

\section{The dual picture in the commutative case}\label{sec2}
\noindent We will now restrict our focus to commutative irreversible algebraic dynamical systems $(G,P,\theta)$: Injective group endomorphisms $\theta_p$ of a discrete abelian group $G$ correspond to surjective group endomorphisms $\hat{\theta}_p$ of its Pontryagin dual $\hat{G}$, which is a compact abelian group. Moreover, the cardinality of $\ker \hat{\theta}_p$ is equal to the index $[G: \theta_p(G)]$. Via duality, we arrive at a definition of (strong) independence for commuting surjective group endomorphisms $\eta_1$ and $\eta_2$ of an arbitrary group $K$, see Definition~\ref{def:ind surj}, which is consistent with \cite{CV}*{Lemma 5.4}. 

With this notion of independence, we then recast the conditions for an irreversible algebraic dynamical system $(G,P,\theta)$ with commutative $G$ in terms of its dual model $(\hat{G},P,\hat{\theta})$, see Proposition~\ref{prop:CIAD as a topological dynamical system}. This provides a new perspective on irreversible algebraic dynamical systems: If $G$ is commutative and $(G,P,\theta)$  is of finite type, it can be regarded as an irreversible topological dynamical system. More precisely, it arises from surjective local homeomorphisms $\hat{\theta}_p$ of the compact Hausdorff space $\hat{G}$, see \cite{Sta2} for details.

We start with a short review of basic facts about characters on groups, see \cite{DE} for details and further information. Recall that a character $\chi$ on a locally compact abelian group $G$ is a continuous group homomorphism $\chi:G \longrightarrow \T$. The set of characters on $G$ forms a locally compact abelian group $\hat{G}$ when equipped with the topology of uniform convergence on compact subsets of $G$. Pontryagin duality states that $\hat{\hat{G}} \cong G$. For this result, we interpret $g \in G$ as a character on $\hat{G}$ via $g(\chi):= \chi(g)$. If $G$ is discrete, then $\hat{G}$ is compact and vice versa.

\begin{definition}\label{def:annihilator}
Let $G$ be a locally compact abelian group. For a subset $H \subset G$, the \emph{annihilator} of $H$ is given by $H^\perp := \{~\chi \in \hat{G}~|~\chi|_H = 1~\}$.
\end{definition}

\begin{remark}\label{rem:annihilator} 
The annihilator is always a closed subgroup of $\hat{G}$. A useful fact about annihilators of subgroups $H$ is that we have $\hat{H} \cong \hat{G}/{H^\perp}$. Additionally, one can show that $(H^\perp)^\perp$ is the smallest closed subgroup of $G$ containing $H$. So if $H \subset G$ is a closed subgroup, then $(H^\perp)^\perp = H$.  
\end{remark}


\begin{lemma}\label{lem:dual picture: converting endomorphisms I}
Let $G$ be a locally compact abelian group and $\eta:G \longrightarrow G$ a group endomorphism. Then $\hat{\eta}(\chi)(g) := \chi{\circ}\eta(g)$ defines a group endomorphism $\hat{\eta}: \hat{G} \longrightarrow \hat{G}$ which is continuous if and only if $\eta$ is and we have: 
\begin{enumerate}[i)]
\item $\hat{\hat{\eta}} = \eta$.
\item $\eta(G)^\perp = \ker\hat{\eta}$.
\item $\hat{\eta}(\hat{G}) \subset \hat{G}$ is dense if and only if $\eta$ is injective.
\item $\widehat{\ker \hat{\eta}} \cong \coker \eta$ if $\eta(G)$ is closed.
\end{enumerate}
\end{lemma}

\noindent In particular, if $G$ is discrete, then ii) states that $\hat{\eta}:\hat{G} \longrightarrow \hat{G}$ is surjective if and only if $\eta:G \longrightarrow G$ is injective. Moreover, $\eta(G)$ is always closed. If, in addition, $\coker \eta$ is finite, then $\ker \hat{\eta} \cong \widehat{\ker\hat{\eta}} \cong \coker \eta$ follows from iv). 

\begin{lemma}\label{lem:dual picture: converting endomorphisms II}
If $G$ is a locally compact abelian group and $H_1,H_2 \subset G$ are subgroups, then: 
\begin{enumerate}[i)]
\item $(H_1 \cdot H_2)^\perp = H_1^\perp \cap H_2^\perp$. 
\item $(H_1 \cap H_2)^\perp = H_1^\perp \cdot H_2^\perp$ holds if $H_1$ and $H_2$ are closed.
\end{enumerate}
\end{lemma}

\begin{proposition}\label{prop:ind corres inj surj}
Let $G$ be a discrete abelian group and $\theta_1,\theta_2$ be commuting, injective endomorphisms of $G$. Then the following statements hold:
\begin{enumerate}[i)]
\item $\theta_1$ and $\theta_2$ are strongly independent if and only if $\ker\hat{\theta}_1$ and $\ker\hat{\theta}_2$ intersect trivially.
\item $\theta_1$ and $\theta_2$ are independent if and only if $\ker\hat{\theta}_1 \cdot \ker\hat{\theta}_2 = \ker\widehat{\theta_1\theta_2}$.
\end{enumerate}
\end{proposition}
\begin{proof}
For strong independence, we compute 
\[(\theta_1(G)\theta_2(G))^\perp \stackrel{\ref{lem:dual picture: converting endomorphisms II}~i)}{=} \theta_1(G)^\perp \cap \theta_2(G)^\perp \stackrel{\ref{lem:dual picture: converting endomorphisms I}~ii)}{=} \ker\hat{\theta}_1 \cap \ker\hat{\theta}_2.\]
Therefore, $\theta_1(G)\theta_2(G) = G$ is equivalent to $\ker\hat{\theta}_1 \cap \ker\hat{\theta}_2 = \{1_{\hat{G}}\}$. Similarly, we get
\[(\theta_1(G) \cap \theta_2(G))^\perp \stackrel{\ref{lem:dual picture: converting endomorphisms II}~ii)}{=} \theta_1(G)^\perp \cdot \theta_2(G)^\perp \stackrel{\ref{lem:dual picture: converting endomorphisms I}~ii)}{=} \ker\hat{\theta}_1 \cdot \ker\hat{\theta}_2.\]
On the other hand, Lemma~\ref{lem:dual picture: converting endomorphisms I}~ii) gives $\ker\widehat{\theta_1\theta_2} = \theta_1\theta_2(G)^\perp$. 
\end{proof}

\noindent This motivates the following definition in analogy to Definition~\ref{def:ind}:

\begin{definition}\label{def:ind surj}
Two commuting, surjective group endomorphisms $\eta_{1}$ and $\eta_2$ of a group $K$ are said to be \emph{strongly independent}, if $\ker\eta_{1}$ and $\ker\eta_{2}$ intersect trivially. $\eta_{1}$ and $\eta_2$ are called \emph{independent}, if $\ker\eta_1 \cdot \ker\eta_2 = \ker\eta_1\eta_2$. 
\end{definition}

\noindent It is clear that we have an equivalence between the statements:
\begin{enumerate}[(i)]\label{ind corres inj surj}
\item $\eta_1$ and $\eta_2$ are strongly independent.
\item $\eta_1$ is an injective group endomorphism of $\ker\eta_2$.
\item[(ii')] $\eta_2$ is an injective group endomorphism of $\ker\eta_1$.
\end{enumerate}
If both $\ker\eta_1$ and $\ker\eta_2$ are finite, then strong independence and independence coincide. Therefore, this definition is consistent with \cite{CV}*{Definition 5.5}, where the case of endomorphisms (of a compact abelian group $K$) with finite kernels is treated. Note that there is no conflict with (strong) independence for injective group endomorphisms, see Definition~\ref{def:ind}, as all these conditions are trivially satisfied by group automorphisms. 

With the observations from Lemma~\ref{lem:dual picture: converting endomorphisms I} and Lemma~\ref{lem:dual picture: converting endomorphisms II} at hands, we can now translate the setup from Definition~\ref{def:IAD} for commutative irreversible algebraic dynamical systems:

\begin{proposition}\label{prop:CIAD as a topological dynamical system}
For a discrete abelian group $G$, a triple $(G,P,\theta)$ is a commutative irreversible algebraic dynamical system if and only if 
\begin{enumerate}[(A)]
\item $\hat{G}$ is a compact abelian group,
\item $P$ is a countably generated, free, abelian monoid (with unit $1_P$), and 
\item $\hat{\theta}$ is an action of $P$ on $\hat{G}$ by surjective group endomorphisms with the property that $\hat{\theta}_p$ and $\hat{\theta}_q$ are independent if and only if $p$ and $q$ are relatively prime in $P$.
\end{enumerate}
$(G,P,\theta)$ is minimal if and only if $\bigcup_{p \in P}{\ker\hat{\theta}_{p}} \subset \hat{G}$ is dense.
It is of finite (infinite) type if and only if $\ker\hat{\theta}_p$ is (infinite) finite for all $p \in P$ ($p \neq 1_P$).
\end{proposition}
\begin{proof}
Conditions (A) and (B) of this characterization follow directly from Lemma~\ref{lem:dual picture: converting endomorphisms I}. Moreover, for any $p \in P$, the equation $(\ker\hat{\theta}_p)^\perp = \im \theta_p$ yields an isomorphism between $\coker \theta_p$ and the Pontryagin dual of $\ker\hat{\theta}_p$. Combining Lemma~\ref{lem:dual picture: converting endomorphisms I}~iii) and Proposition~\ref{prop:ind corres inj surj} yields (C). Note that we have $\theta_q(G) \subset \theta_p(G)$ and, correspondingly, $\ker\hat{\theta}_p \subset \ker\hat{\theta}_q$ whenever $q \in pP$. Since $P$ is directed, Lemma~\ref{lem:dual picture: converting endomorphisms II}~i) and Lemma~\ref{lem:dual picture: converting endomorphisms I}~ii) yield the equivalence between minimality of $(G,P,\theta)$ and $\bigcup_{p \in P}{\ker\hat{\theta}_{p}}$ being dense in $\hat{G}$.  
For the last claim, we recall that a locally compact abelian group is finite if and only if its dual group is finite. Thus $\ker\hat{\theta}_p$ is finite if and only if $\coker\theta_p$ is finite. 
\end{proof}

\noindent We will now revisit some of the examples from Section~1 to present their dual models:

\begin{examples}\label{ex:dual IAD standard}
The following list corresponds to the one in Example~\ref{ex:IAD standard}.
\begin{enumerate}[(a)]
\item For $G = \Z$, a family of relatively prime numbers $(p_i)_{i \in I} \subset \Z^\times{\setminus}\Z^*$ generates a monoid $P = |(p_i)_{i \in I}\rangle \subset \Z^\times$ which acts by $\theta_{p_i}(g) = p_ig$. In this case, $\hat{G} = \IT$ and $\hat{\theta}_p(t) = t^p$ for all $t \in \IT$ and $p \in P$.
\item For $I\subset\N, 0 \in I$, let $q, (p_i)_{i \in I} \subset \Z^\times\setminus \Z^*$ be relatively prime numbers and set $P = |(p_i)_{i \in I}\rangle$ as well as $G = \Z[1/q] = \varinjlim \Z$ with connecting maps given by multiplication with $q$. Then this constitutes a minimal commutative irreversible algebraic dynamical system of finite type, see Example~\ref{ex:IAD standard}~(b). Then $\hat{G}$ is the solenoid $\Z_q = \varprojlim \Z/q^k\Z$, on which $\hat{\theta}_p$ is given by multiplication with $p$.
\item For a finite field $\IK$, let $p_i \in \IK[T], i \in I$ (for an index set $I$) be polynomials in $G=\IK[T]$ with the property that $(p_i) \cap (p_j) = (p_ip_j)$ holds for all $i \neq j$. Then the action $\theta$ of $P := |(p_i)_{i \in I}\rangle$ given by multiplication with the polynomial itself yields a commutative irreversible algebraic dynamical system of finite type, see Example~\ref{ex:IAD standard}~(c). Then $\hat{G}$ is the ring of formal power series $\IK[[T]]$ over $\IK$, compare \cite{CV}*{Example 2.1.4}, and $\hat{\theta}_p$ is given by multiplication with $p$ in $\IK[[T]]$.   
\end{enumerate}
\end{examples}

\begin{example}\label{ex:dual IAD integer matrices}
Recall that, in Example~\ref{ex:IAD integer matrices}, we considered $G=\Z^d$ for some $d \geq 1$, a family of pairwise commuting matrices $(T_i)_{i \in I} \subset M_d(\Z) \cap Gl_d(\IQ)$ satisfying $|\det T_i| > 1$ for all $i \in I$ and set $P = |(T_i)_{i \in I}\rangle$ with $\theta_{T_i}(g) = T_ig$. In this case, we have $\hat{G} = \IT^d$ and the endomorphism $\hat{\theta}_p$ is given by the matrix corresponding to $\theta_p$ interpreted as an endomorphism of $\R^d/\Z^d \cong \IT^d$.
\end{example}


\begin{example}\label{ex:dual IAD for On}
The dual model for the unilateral shift on $G= \bigoplus_\N \Z/n\Z$ for $n \geq 2$ from Example~\ref{ex:IAD for On} is given by the shift $(x_k)_{k \in \N} \mapsto (x_{k+1})_{k \in \N}$ on $\hat{G} = (\Z/n\Z)^\N$. The discussion for Example~\ref{ex:ind doesn't give strong ind} with the restriction that $G_0$ be abelian is analogous, where $\N$ is replaced by $P$ and $\Z/n\Z$ by $G_0$.
\end{example}

\begin{example}\label{ex:dual CIAD as sum of CIADs}
In the situation of Example~\ref{ex:IAD as sum of IADs}, where we will now require that $(G_n,P,\theta^{(i)})_{i\in \N}$ be a family of commutative irreversible algebraic dynamical systems, $G = \bigoplus_{i \in \N} G_i$ turns into $\hat{G} = \prod_{i \in \N} \hat{G}_i$. For each $p \in P$, the group endomorphism $\hat{\theta}_p$ is given by applying $\theta^{(i)}_p$ to the $i$-th component of $\hat{G}$. $\ker\hat{\theta}_p$ is infinite for all $p \in P\setminus\{1_P\}$. If each $\theta^{(i)}$ satisfies the strong independence condition from Definition~\ref{def:ind}, $\hat{\theta}$ satisfies the strong independence condition from Definition~\ref{def:ind surj} due to Proposition~\ref{prop:ind corres inj surj}.
\end{example}

\section{Structure of the associated C*-algebras}\label{sec3}
\noindent In this section, we associate a universal C*-algebra $\CO[G,P,\theta]$ to every irreversible algebraic dynamical system $(G,P,\theta)$. The general approach is inspired by the methods of \cite{CV} for the case of a single group endomorphism with finite cokernel of a discrete abelian group. Partly, these ideas can even be traced back to \cite{Cun}. Note however, that we are going to use a different spanning family than the one used in \cite{CV}.

We will examine structural properties of $\CO[G,P,\theta]$ as well as of two nested subalgebras: the core $\CF$ and the diagonal $\CD$. In Lemma~\ref{lem:IAD spec D as a completion of G}, a description of the spectrum $G_\theta$ of the diagonal $\CD$ is provided, which allows us to regard $G_\theta$ as a completion of $G$ with respect to $\theta$ in the case where $(G,P,\theta)$ is minimal, compare \cite{CV}*{Lemma 2.4}. 

Based on the description of $G_\theta$, the action $\hat{\tau}$ of $G$ on $G_\theta$ coming from $\tau_g (e_{h,p}) = e_{gh,p}$ is shown to be always minimal. Moreover, we prove that topological freeness of $\hat{\tau}$ corresponds to minimality of $(G,P,\theta)$, see Proposition~\ref{prop:IAD min iff tau min+topfree}. As an immediate consequence we deduce that $\CD \rtimes_\tau G$ is simple if and only if $(G,P,\theta)$ is minimal and $\hat{\tau}$ is amenable, see Corollary~\ref{cor:D rtimes G simple iff IAD minimal}. This crossed product is actually isomorphic to $\CF$, see Corollary~\ref{cor:F cong D rtimes G-general IAD}. 

We remark that our strategy of proof differs from the one of \cite{CV} because we start by establishing an isomorphism between $\CO[G,P,\theta]$ and $\CD \rtimes (G\rtimes_\theta P)$, compare Proposition~\ref{prop:O cong D rtimes GxP} and \cite{CV}*{Lemma 2.5 and Theorem 2.6}. By Theorem~\ref{thm:cr prod by sd prod as it cr prod}, we deduce that $\CO[G,P,\theta]$ is isomorphic to the semigroup crossed product $\CF \rtimes P$. So we get 
\[\CO[G,P,\theta] \cong \CD \rtimes (G\rtimes_\theta P) \cong \CF \rtimes P.\]
One advantage of this strategy is that we are able to establish these isomorphisms in greater generality, i.e. without minimality of $(G,P,\theta)$ and amenability of $\hat{\tau}$ which would give simplicity of both $\CF$ and $\CO[G,P,\theta]$. 

Similar to \cite{CV}, we conclude that, whenever $(G,P,\theta)$ is minimal and the $G$-action $\hat{\tau}$ on $G_\theta$ is amenable, the C*-algebra $\CO[G,P,\theta]$ is a unital UCT Kirchberg algebra, see Theorem~\ref{thm:O p.i. and simple} and Corollary~\ref{cor:UCT Kirchberg algebra}. Thus $\CO[G,P,\theta]$ is classified by its K-theory in this case due to the important classification results of Christopher Phillips and Eberhard Kirchberg, see \cite{Kir}.\vspace*{4mm}

\noindent Throughout this section, $(G,P,\theta)$ will represent an irreversible algebraic dynamical system unless specified otherwise. Let $(\xi_g)_{g \in G}$ denote the canonical orthonormal basis of $\ell^2(G)$. For $g \in G$ and $p \in P$, define operators $U_g$ and $S_p$ on $\ell^2(G)$ by $U_g(\xi_{g'}) := \xi_{gg'}$ and $S_p(\xi_{g'}) := \xi_{\theta_p(g')}$ for $g' \in G$. Then $(U_g)_{g \in G}$ is a unitary representation of the group $G$ and $S_p^*(\xi_{g'}) = \chi_{\theta_p(G)}(g') \xi_{\theta_p^{-1}(g')}$ for $g' \in G$, so $(S_p)_{p \in P}$ is a representation of the semigroup $P$ by isometries. Furthermore, these operators satisfy 
\[\begin{array}{lrclcl}
(\text{CNP }1) & S_{p}U_{g}(\xi_{g'}) &=& \xi_{\theta_p(gg')} &=& U_{\theta_{p}(g)}S_{p}(\xi_{g'}),\\ 
&&\text{ and}\\
(\text{CNP }3) & \sum\limits_{[g] \in G/\theta_{p}(G)}{E_{g,p}}(\xi_{g'}) &=& \xi_{g'} &&\text{ if } [G : \theta_{p}(G)]< \infty,
\end{array}\vspace*{-1mm}\]
where $E_{g,p} = U_{g}S_{p}S_{p}^{*}U_{g}^{*}$. In fact, (CNP $3$) holds also in the case of an infinite index $[G : \theta_{p}(G)]$, as $(\sum_{[g] \in F}E_{g,p})_{F \subset G/\theta_p(G)}$ converges to the identity on $\ell^2(G)$ as $F \nearrow G/\theta_p(G)$ with respect to the strong operator topology. But this convergence does not hold in norm because each $E_{g,p}$ is a non-zero projection. In view of our motivation to construct a universal C*-algebra based on this model, it is therefore reasonable to restrict this relation to the case where $[G:\theta_p(G)]$ is finite.

As the numbering indicates, we are interested in an additional relation (CNP $2$) which will increase the accessibility of the universal model: If $G$ was trivial, this would simply be the condition that $S_p$ and $S_q$ doubly commute for all relatively prime $p$ and $q$ in $P$, i.e. $S_p^*S_q = S_qS_p^*$. This condition has been employed successfully for quasi-lattice ordered groups, see \cite{Nic}*{Section 3} and also \cite{LR2} for more information. But as $G$ is an infinite group, this will not be sufficient. 

Moreover, we want to ensure that, within the universal model to be built, an expression corresponding to $S_p^*U_gS_p$ belongs to $C^*(G)$. This property has been used extensively in the context of semigroup crossed products involving transfer operators, see \cites{Exe1,Lar}.

An entirely different way to put it is that we aim for a better understanding of the structure of the commutative subalgebra $C^*(\{~E_{g,p}~|~g \in G,~p \in P~\})$ inside $\CL(\ell^2(G))$. In a much more general framework, this has been considered by Xin Li, see \cite{Li1} and resulted in a new definition of semigroup C*-algebras for discrete left cancellative semigroups with identity. One particular strength of his notion is the close connection between amenability of semigroups and nuclearity of their C*-algebras, see \cite{Li2}. 

All of these three instances suggest that a closer examination of the terms $S_p^*U_gS_q$ is in order. For $g = \theta_p(g_1)\theta_q(g_2)$ with $g_1,g_2 \in G$, we get $S_p^*U_gS_q = U_{g_1}S_{(p \wedge q)^{-1}q}S_{(p \wedge q)^{-1}p}^{*}U_{g_2}$. On the other hand, $g \notin \theta_p(G)\theta_q(G)$ is equivalent to $g\theta_q(G) \cap \theta_p(G) = \emptyset$, which forces $S_p^*U_gS_q = 0$. Thus we get
\[\begin{array}{lrclcl}
(\text{CNP }2) &\hspace*{-2mm} S_{p}^{*}U_gS_{q} &\hspace*{-2.5mm}=\hspace*{-2.5mm}& \begin{cases} 
U_{g_1}S_{(p \wedge q)^{-1}q}S^{*}_{(p \wedge q)^{-1}p}U_{g_2}& \text{ if } g = \theta_p(g_1)\theta_q(g_2),\\ 0& \text{ else.}\end{cases}
\end{array}\]
for all $g \in G,p,q \in P$. These observations motivate the following definition:

\begin{definition}\label{def:O-algebra ad-hoc}
$\CO[G,P,\theta]$ is the universal C*-algebra generated by a unitary representation $(u_{g})_{g \in G}$ of the group $G$ and a representation $(s_{p})_{p \in P}$ of the semigroup $P$ by isometries subject to the relations:
\[\begin{array}{lrcl}
(\text{CNP }1) & s_{p}u_{g} &\hspace*{-2.5mm}=\hspace*{-2.5mm}& u_{\theta_{p}(g)}s_{p}\vspace*{2mm}\\
(\text{CNP }2) & s_{p}^{*}u_gs_{q} &\hspace*{-2.5mm}=\hspace*{-2.5mm}& \begin{cases} 
u_{g_1}s_{(p \wedge q)^{-1}q}s_{(p \wedge q)^{-1}p}^{*}u_{g_2}& \text{ if } g = \theta_p(g_1)\theta_q(g_2),\\ 0,& \text{ else.}\end{cases}\vspace*{2mm}\\
(\text{CNP }3) & 1 &\hspace*{-2.5mm}=\hspace*{-2.5mm}& \sum\limits_{[g] \in G/\theta_{p}(G)}{e_{g,p}} \hspace*{2mm}\text{ if } [G : \theta_{p}(G)]< \infty,
\end{array}\]
where $e_{g,p} = u_{g}s_{p}s_{p}^{*}u_{g}^{*}$.
\end{definition}

\noindent We note the following immediate consequence of the construction:

\begin{proposition}\label{prop:O-canonical rep}
$\CO[G,P,\theta]$ has a canonical non-trivial representation on $\ell^{2}(G)$ given by $u_g \mapsto U_g,~s_p \mapsto S_p$. In particular, $\CO[G,P,\theta]$ is non-zero.
\end{proposition}

\begin{remark}\label{rem:O[G,P,theta]}~
\begin{enumerate}[a)]
\item The presence of (CNP 1) guarantees that the expression in (CNP 2) is independent of the choice of $g_1$ and $g_2$ satisfying $g = \theta_p(g_1)\theta_q(g_2)$. To see this, suppose $g_3$ and $g_4$ satisfy $g = \theta_p(g_3)\theta_q(g_4)$ as well. Since $G$ is a group, $\theta_p(g_1^{-1}g_3) = \theta_q(g_2g_4^{-1})$ follows. This is equivalent to $\theta_{p'}(g_1^{-1}g_3) = \theta_{q'}(g_2g_4^{-1})$ for $p' := (p \wedge q)^{-1}p$ and $q':= (p \wedge q)^{-1}q$ by injectivity of $\theta_{p \wedge q}$. As $p'$ and $q'$ are relatively prime, condition (C) from Definition~\ref{def:IAD} implies $g_1^{-1}g_3 \in \theta_{q'}(G)$ and $g_2g_4^{-1} \in \theta_{p'}(G)$. Injectivity of $\theta_{p \vee q}$ and $\theta_{p'}(g_1^{-1}g_3)\theta_{q'}(g_4g_2^{-1}) = 1_G$ yield $\theta_{q'}^{-1}(g_1^{-1}g_3)\theta_{p'}^{-1}(g_2g_4^{-1}) = 1_G$.
Hence we conclude
\[\begin{array}{lcl}
u_{g_3}s_{q'}s_{p'}^{*}u_{g_4} 
&=& u_{g_1}u_{g_1^{-1}g_3}s_{q'}s_{p'}^{*}u_{g_4g_2^{-1}}u_{g_2}\vspace*{2mm}\\
&=& u_{g_1}s_{q'}u_{\theta_q^{-1}(g_1^{-1}g_3)\theta_p^{-1}(g_4g_2^{-1})}s_{p'}^*u_{g_2}\vspace*{2mm}\\
&=& u_{g_1}s_{q'}s_{p'}^{*}u_{g_2}.
\end{array}\]
\item If $p \in P$ and $g_{1},g_{2} \in G$ satisfy $g_{1}\theta_{p}(G) = g_{2}\theta_{p}(G)$, then 
\[e_{g_{2},p} = u_{g_1}u_{g_1^{-1}g_2}s_ps_p^*u_{g_2}^* = u_{g_1}s_ps_p^*u_{g_1^{-1}g_2}u_{g_2}^* = e_{g_{1},p}\]
follows from (CNP $1$). Thus the summation in (CNP $3$) makes sense.
\item Condition (CNP $2$) includes the following two special cases:
\[\begin{array}{lcll}
\hspace*{8mm} s_p^*s_q &=& s_qs_p^* &\text{ for all relatively prime } p,q \in P.\\
\hspace*{8mm} s_p^*u_gs_p &=& \chi_{\theta_p(G)}(g) u_{\theta_p^{-1}(g)} &\text{ for all } g \in G, p \in P.
\end{array}\]  
\end{enumerate} 
\end{remark}

\noindent One immediate benefit of  (CNP $1$) and  (CNP $2$) is the following lemma, whose straightforward proof is omitted.

\begin{lemma}\label{lem:dense span}
The linear span of $(u_gs_ps_q^{*}u_h)_{g,h \in G, p,q \in P}$ is dense in $\CO[G,P,\theta]$.
\end{lemma}

\begin{lemma}\label{lem:e_(g,p) commute}
The projections $(e_{g,p})_{g \in G, p \in P}$ commute. More precisely, for $g,h \in G$ and $p,q \in P$, we have \[e_{g,p}e_{h,q} = 
\begin{cases} 
e_{g\theta_{p}(h'),p \vee q}& \text{ if } g^{-1}h \in \theta_p(G)\theta_q(G),\\
0& \text{ else,}
\end{cases}\]
where $h' \in G$ is determined uniquely up to multiplication from the right by elements of $\theta_{p^{-1}(p \vee q)}(G)$ by the condition that $g\theta_{p}(h') \in h\theta_{q}(G)$.
\end{lemma}
\begin{proof}
For $g,h \in G$ and $p,q \in P$, the product $e_{g,p}e_{h,q}$ is non-zero only if $g^{-1}h \in \theta_p(G)\theta_q(G)$ by (CNP $2$). So let us assume that $g^{-1}h \in \theta_p(G)\theta_q(G)$ holds. Then there are $g',h' \in G$ such that $g^{-1}h = \theta_p(h')\theta_q(g')$. As $G$ is a group, this is equivalent to $h\theta_q(g')^{-1} = g\theta_p(h')$. Thus we get 
\[e_{g,p}e_{h,q} = u_{g\theta_p(h')}s_ps_{(p \wedge q)^{-1}q}s_{(p \wedge q)^{-1}p}^*s_q^*u^*_{h\theta_q(g')^{-1}} 
= e_{g\theta_p(h'),p \vee q}.\]  
Clearly, this also proves that the two projections commute. The uniqueness assertion follows from (CNP $2$). 
\end{proof}

\begin{definition}\label{def:diagonal}
The C*-subalgebra $\CD$ of $\CO[G,P,\theta]$ generated by the commuting projections $(e_{g,p})_{g \in G, p \in P}$ is called the \emph{diagonal}. In addition, let $\CD_p := C^*(\{e_{g,q} \mid [g] \in G/\theta_p(G),p \in qP\})$ denote the C*-subalgebra of $\CD$ corresponding to $p \in P$.
\end{definition} 

\noindent We note the following obvious fact:

\begin{lemma}\label{lem:D as inductive limit}
For all $p,q \in P$, $p \in qP$ implies $\CD_q \subset \CD_p$. $\CD$ is the closure of $\bigcup_{p \in P} \CD_p$. If $[G : \theta_p(G)]$ is finite, then 
\[\CD_p = \Span\{e_{g,p} \mid [g] \in G/\theta_p(G)\} \cong \C^{[G:\theta_p(G)]}.\]
\end{lemma}

\noindent Let us make the following non-trivial observation:

\begin{lemma}\label{lem:small subprojections1}
Suppose $g \in G, p \in P$ and a finite subset $F$ of $G \times P$ are chosen in such a way that $e_{g,p}\prod_{(h,q) \in F}(1-e_{h,q})$ is non-zero. Then there exist $g' \in G$ and $p' \in P$ satisfying $e_{g',p'} \leq e_{g,p}\prod_{(h,q) \in F}(1-e_{h,q})$.
\end{lemma}
\begin{proof}
If $F$ is empty, then $\prod_{(h,q) \in F}(1-e_{h,q}) = 1$ by convention, so there is nothing to show. Now let $F$ be non-empty. For $(h,q) \in F$, let us decompose $q$ uniquely as $q = q^{(fin)}q^{(inf)}$, where $[G:\theta_{q^{(fin)}}(G)]$ is finite and we require that, for each $r \in P$ with $q \in rP$, finiteness of $[G:\theta_r(G)]$ implies $q^{(fin)} \in rP$. In other words, $[G:\theta_r(G)]$ is infinite for every $r \neq 1_P$ with $q^{(inf)} \in rP$. Using (CNP 3) for $q^{(fin)}$ and Lemma~\ref{lem:e_(g,p) commute}, we compute
\[\begin{array}{lcl}
1-e_{h,q} &=& (1-e_{h,q^{(fin)}}e_{h,q^{(inf)}})\sum\limits_{[k] \in G/\theta_{q^{(fin)}}(G)} e_{k,p^{(fin)}}\vspace*{2mm}\\
&=& e_{h,q^{(fin)}}(1-e_{h,q^{(inf)}}) + \sum\limits_{\substack{[k] \in G/\theta_{q^{(fin)}}(G)\\ [k] \neq [h]}} e_{k,q^{(fin)}}.
\end{array}\] 
Therefore, we can rewrite the initial product as
\[\begin{array}{c} 
e_{g,p}\prod\limits_{(h,q) \in F}(1-e_{h_i,q_i}) = \sum\limits_{(\tilde{g},\tilde{p}) \in \tilde{F}} e_{\tilde{g},\tilde{p}}\prod\limits_{(h,q) \in F_{(\tilde{g},\tilde{p})}}(1-e_{h,q}),
\end{array}\]
where
\begin{enumerate}[$\cdot$]
\item $\tilde{F}$ is a finite subset of $G{\times}P$,
\item $e_{\tilde{g},\tilde{p}} \leq e_{g,p}$ for all $(\tilde{g},\tilde{p}) \in \tilde{F}$,
\item the projections $(e_{\tilde{g},\tilde{p}})_{(\tilde{g},\tilde{p}) \in \tilde{F}}$ are mutually orthogonal,
\item for each $(\tilde{g},\tilde{p}) \in \tilde{F}$, $F_{(\tilde{g},\tilde{p})}$ is a finite subset of $G{\times}P$, and
\item each $(h,q) \in F_{(\tilde{g},\tilde{p})}$ satisfies $q = q^{(inf)}$ and $\tilde{p} \notin qP$.
\end{enumerate}
Since the product $e_{g,p}\prod_{(h,q) \in F}(1-e_{h_i,q_i})$ on the left hand side is non-zero, there is $(g_0,p_0) \in \tilde{F}$ such that $e_{g_0,p_0}\prod_{(h,q) \in F_{(g_0,p_0)}}(1-e_{h,q})$ is non-zero. Without loss of generality, we may assume that $e_{g_0,p_0}e_{h,q}$ is non-zero for all $(h,q) \in F_{(g_0,p_0)}$. Consider $F_P := \{p_0 \vee q \mid (h,q) \in F_{(g_0,p_0)} \text{ for some } h \in G\}$. Pick $p_1 \in F_P$ which is minimal in the sense that for any other $r \in F_P$, $p_1 \in rP$ implies $r=p_1$. Let $(h_1,q_1),\dots,(h_n,q_n) \in F_{(g_0,p_0)}$ denote the elements satisfying $p_0 \vee q_i = p_1$. According to Lemma~\ref{lem:e_(g,p) commute}, we have 
\[e_{g_0,p_0}e_{h_i,q_i} = e_{g_0\theta_{p_0}(g'_i),p_1} \text{ for a suitable } g'_i \in G \text{ (for } i=1,\dots,n).\] 
Note that $p_0^{-1}p_1 \neq 1_P$ and $q_1 = q_1^{(inf)} \in p_0^{-1}p_1P$, so $[G:\theta_{p_0^{-1}p_1}(G)]$ is infinite. Hence there exists $g_1 \in g_0\theta_{p_0}(G)$ with 
\[e_{g_1,p_1} \leq e_{g_0,p_0} \text{ and } e_{g_1,p_1}e_{h_i,q_i} = 0 \text{ for } i = 1,\dots,n.\]
Setting 
\[F_{(g_1,p_1)}:= \{(h,q) \in F_{(g_0,p_0)} \mid e_{h,q}e_{g_1,p_1} \neq 0\} \subsetneqq F_{(g_0,p_0)},\]
we observe that 
\[\begin{array}{c} e_{g_1,p_1} \prod\limits_{(h,q) \in F_{(g_1,p_1)}}(1-e_{h,q}) \neq 0 \end{array}\]
follows from the initial statement for $(g_0,p_0)$ and $F_{(g_0,p_0)}$ since we have chosen $p_1$ in a minimal way. Indeed, if the product was trivial, then there would be $(h,q) \in F_{(g_1,p_1)}$ with $e_{h,q} \geq e_{g_1,p_1}$. By Lemma~\ref{lem:e_(g,p) commute}, this would force $p_1 \in qP$ and therefore $p_1 \in (p_1 \vee q)P \subset (p_0 \vee q)P$, which cannot be true since $p_1$ was chosen in a minimal way. 

Thus, we can iterate the process used to obtain $(g_1,p_1)$ and $F_{(g_1,p_1)}$ for $(g_0,p_0)$ and $F_{(g_0,p_0)}$. After finitely many steps, we arrive at an element $(g_n,p_n)=:(g',p')$ with the property that $e_{g',p'} \leq e_{g_0,p_0}$ is orthogonal to $e_{h,q}$ for all $(h,q) \in F_{(g_0,p_0)}$. This establishes the claim.
\end{proof}

\noindent The possibility of passing to smaller subprojections that avoid finitely many defect projections provided through Lemma~\ref{lem:small subprojections1} will be crucial for the proof of pure infiniteness and simplicity of $\CO[G,P,\theta]$, see Theorem~\ref{thm:O p.i. and simple} and in particular Lemma~\ref{lem:small subprojections2}. A first application of this observation lies in the determination of the spectrum of $\CD$:

\begin{lemma}\label{lem:IAD spec D as a completion of G}
The spectrum of $\CD$, denoted by $G_\theta$, is a totally disconnected, compact Hausdorff space. A basis for the topology on $G_\theta$ is given by the cylinder sets
\[Z_{(g,p),(h_1,q_1),\dots,(h_n,q_n)} = \{\chi \in G_\theta \mid \chi(e_{g,p}) = 1,~\chi(e_{h_i,q_i}) = 0 \text{ for all } i\},\]
where $n \in \N, g,h_1,\dots,h_n \in G$ and $p,q_1,\dots,q_n \in P$. Moreover, 
\[\iota(g) \in Z_{(g',p),(h_1,q_1),\dots,(h_n,q_n)} \Longleftrightarrow g \in g'\theta_p(G) \text{ and } g \notin h_i\theta_{q_i}(G) \text{ for all } i\]
defines a map $\iota:G \longrightarrow G_\theta$ with dense image. $\iota$ is injective if and only if $(G,P,\theta)$ is minimal.
\end{lemma}
\begin{proof}
$G_\theta$ is a totally disconnected, compact Hausdorff space since $\CD$ is a unital C*-algebra generated by commuting projections. The statement concerning the basis for the topology on $G_\theta$ follows from Lemma~\ref{lem:D as inductive limit}. To see that $\iota$ has dense image, let $\chi \in G_\theta$. As the cylinder sets form a basis for the topology of $G_\theta$, every open neighbourhood of $\chi$ contains a cylinder set $Z_{(g,p),(h_1,q_1),\dots,(h_n,q_n)}$ with $\chi \in Z_{(g,p),(h_1,q_1),\dots,(h_n,q_n)}$. This means that $e_{g,p}\prod_{i=1}^n (1-e_{h_i,q_i})$ is non-zero. Hence we can apply Lemma~\ref{lem:small subprojections1} to obtain $(g',p') \in G{\times}P$ satisfying $e_{g',p'} \leq e_{g,p}\prod_{i=1}^n (1-e_{h_i,q_i})$. In other words, $\iota(g') \in Z_{(g,p),(h_1,q_1),\dots,(h_n,q_n)}$, so $\iota(G)$ is a dense subset of $G_\theta$. Now given $g,h \in G$, we observe that $\iota(g)=\iota(h)$ is equivalent to $g^{-1}h \in \bigcap_{p \in P}\theta_p(G)$ because the cylinder sets form a basis of the topology on the Hausdorff space $G_\theta$. Therefore $\iota$ is injective precisely if $(G,P,\theta)$ is minimal.
\end{proof}

\begin{remark}\label{rem:IAD SpecD as completion}
By the preceding lemma, $G_\theta$ is a completion of $G$ with respect to $\theta$ whenever $(G,P,\theta)$ is minimal. 
\end{remark}

\noindent There is a canonical action $\tau$ of $G$ on $\CD$ given by $\tau_g(e_{h,p}) = e_{gh,p}$ for $g,h \in G$ and $p \in P$. Known results, as for instance \cite{CV}*{Lemma 2.5}, indicate that $\CD \rtimes_\tau G$ ought to be simple provided that the irreversible algebraic dynamical system $(G,P,\theta)$ is minimal. Of course, this can only be true if the $G$-action $\tau$ on $\CD$ is regular, that is, $\CD \rtimes_\tau G \cong \CD \rtimes_{\tau,r} G$ via the canonical map. Building on the results of \cite{AD}, this can be rephrased as amenability of the action $\hat{\tau}$ on $G_\theta$, see also \cite{BO}*{Theorem 4.4.3} for a concise exposition. Moreover, the map $\iota$ from Lemma~\ref{lem:IAD spec D as a completion of G} satisfies $\hat{\tau}_g(\iota(h)) = \iota(gh)$ for all $g,h \in G$.

If $\hat{\tau}$ is amenable, the celebrated result of \cite{AS} states that the crossed product $C(G_\theta)\rtimes_\tau G$ is simple if and only if the action $\hat{\tau}$ is minimal and topologically free. As it turns out, minimality of $(G,P,\theta)$ corresponds precisely to these two properties. For convenience, let us recall the standard definitions of topological freeness and minimality for group actions.

\begin{definition}\label{def:top free}
Let $X$ be a topological space and $G$ a group. A $G$-action on $X$ is said to be \emph{topologically free}, if the set $X^g = \{x \in X \mid g.x = x\}$ has empty interior for $g \in G\setminus\{1_G\}$.
\end{definition}

\begin{definition}\label{def:min}
Let $X$ be a topological space and $G$ a group. A $G$-action on $X$ is said to be \emph{minimal}, if the orbit $\CO(x) = \{g.x \mid g \in G\}$ is dense in $X$ for every $x \in X$.
\end{definition}

\noindent Equivalently, an action is minimal if the only invariant open (closed) subsets of $X$ are $\emptyset$ and $X$.

\begin{proposition}\label{prop:IAD min iff tau min+topfree}
If $(G,P,\theta)$ is an irreversible algebraic dynamical system, then the action $G$-action $\hat{\tau}$ on $G_\theta$ is minimal. It is topologically free if and only if $(G,P,\theta)$ is minimal.
\end{proposition}
\begin{proof}
On $\iota(G)$, which is dense in $G_\theta$ by Lemma~\ref{lem:IAD spec D as a completion of G}, $\hat{\tau}$ is simply given by translation from the left. Hence $\hat{\tau}$ is minimal. For the second part, we note that $\tau_g = \text{id}_\CD$ holds for every $g \in \bigcap_{p \in P} \theta_p(G)$. Thus, if $(G,P,\theta)$ is not minimal, there is $g \neq 1_G$ such that $G_\theta^g = G_\theta$, so $\hat{\tau}$ is not topologically free. If $(G,P,\theta)$ is minimal, then $\hat{\tau}$ acts freely on $\iota(G)$ because $\iota$ is injective and $G$ is left-cancellative. Since $\iota(G)$ is dense in $G_\theta$, we conclude that $\hat{\tau}$ is topologically free.
\end{proof}

\begin{corollary}\label{cor:D rtimes G simple iff IAD minimal}
The crossed product $\CD \rtimes_\tau G$ is simple if and only if $(G,P,\theta)$ is minimal and $\hat{\tau}$ is amenable.
\end{corollary}
\begin{proof}
Due to a central result from \cite{AD}, amenability of the action is equivalent to regularity of the crossed product. Hence Proposition~\ref{prop:IAD min iff tau min+topfree} and \cite{AS}*{Corollary following Theorem 2} establish the claim. 
\end{proof}

\begin{definition}\label{def:core}
The \emph{core} $\CF$ is the C*-subalgebra of $\CO[G,P,\theta]$ generated by $\CD$ and $(u_{g})_{g \in G}$.
\end{definition}

\begin{lemma}\label{lem:F dense span family}
 The linear span of $(u_gs_ps_p^*u_h^*)_{g,h \in G,p \in P}$ is dense in $\CF$.
\end{lemma}
\begin{proof}
This follows immediately from the calculations for Lemma~\ref{lem:dense span}.
\end{proof}

\begin{remark}\label{rem:E_1}
For every irreversible algebraic dynamical system $(G,P,\theta)$, $P$ is a discrete abelian Ore semigroup. Therefore its enveloping group $P^{-1}P$ is discrete abelian. Let us denote the dual group of $P^{-1}P$ by $L$, which is a compact abelian group by Pontryagin duality. Furthermore, $L$ acts on $\CO[G,P,\theta]$ via the so-called gauge-action $\gamma$ given by 
\[\gamma_{\ell}(u_{g}) = u_{g} \text{ and } \gamma_{\ell}(s_{p}) = \ell(p) s_{p}, \text{ for } g \in G,p\in P \text{ and } \ell \in L.\] 
\begin{enumerate}[a)]
\item The fixed-point algebra of $\gamma$ coincides with $\CF$.
\item If $\mu$ denotes the normalized Haar measure on $L$, then 
\[\begin{array}{c} E_{1}(u_gs_ps_q^*u_h^*) := \int\limits_{\ell \in L}{\gamma_{\ell}(u_gs_ps_q^*u_h^*) d\mu(\ell)} = \delta_{p q} u_gs_ps_p^*u_h^* \end{array}\]
defines a faithful conditional expectation $E_{1}: \CO[G,P,\theta] \longrightarrow \CF$ as $\gamma$ is strongly continuous.
\end{enumerate} 
\end{remark}

\noindent The similarity between $\CF$ and $\CD \rtimes_\tau G$ is apparent. If one assumes that $\CD \rtimes_\tau G$ is simple, which by Corollary~\ref{cor:D rtimes G simple iff IAD minimal} means that the irreversible algebraic dynamical system $(G,P,\theta)$ is minimal, it is easy to show that these two algebras are isomorphic. This strategy has been pursued in \cite{CV}*{Lemma~2.5}. 

However, we will show in Corollary~\ref{cor:F cong D rtimes G-general IAD} that this identification holds in full generality. To do so, we will first derive a semigroup crossed product description $\CO[G,P,\theta] \cong \CD \rtimes (G \rtimes_\theta P)$, which is of independent interest, compare \cite{CV}*{Theorem~2.6}. Also, if $(G,P,\theta)$ is of infinite type, that is, $[G:\theta_p(G)]$ is infinite for all $p \neq 1_P$, then this result reproduces the standard picture $C^*(S) \cong \CD_S \rtimes S$ for C*-algebras of left cancellative semigroups $S$ in the case where $S = G \rtimes_\theta P$, compare \cite{Li1}*{Lemma~2.14}.  

In order to get down to $\CF$ and $\CD \rtimes_\tau G$, respectively, we observe that a crossed product coming from a semidirect product of discrete semigroups can be displayed as an iterated semigroup crossed product under a certain condition, see Theorem~\ref{thm:cr prod by sd prod as it cr prod}. This condition will be satisfied as $G$ is a group, see Remark~\ref{rem:ext to proper endomorphisms for non-unital coeff algs}~b).

\begin{proposition}\label{prop:O cong D rtimes GxP}
Let $(v_{(g,p)})_{(g,p) \in G{\rtimes_\theta}P}$ denote the family of isometries in $\CD{\rtimes}(G{\rtimes_\theta}P)$ implementing the action of the semigroup $G\rtimes_\theta P$ on $\CD$ given by $(g,p).e_{h,q} = e_{g\theta_p(h),pq}$, that is, $v_{(g,p)}e_{h,q}v_{(g,p)}^* = e_{g\theta_p(h),pq}$. Then the map 
\[\begin{array}{rcl}
\CO[G,P,\theta] &\stackrel{\varphi}{\longrightarrow}& \CD{\rtimes}(G{\rtimes_\theta}P)\\
u_gs_p&\mapsto&v_{(g,p)}
\end{array}\] is an isomorphism.
\end{proposition}
\begin{proof}
Recall from Definition~\ref{def:O-algebra ad-hoc} that $\CO[G,P,\theta]$ is the universal C*-algebra generated by a unitary representation $(u_g)_{g \in G}$ of the group $G$ and a semigroup of isometries $(s_p)_{p \in P}$ subject to the relations (CNP 1)--(CNP 3). Hence, in order to show that $\varphi$ defines a surjective $\ast$-homomorphism, it suffices to show that for every $g \in G$, the isometry $v_{(g,1_P)}$ is a unitary, and that the families $(v_{(g,1_P)})_{g \in G},(v_{(1_G,p)})_{p \in P}$ satisfy (CNP 1)--(CNP 3):
\[\begin{array}{lrcl}
&\hspace*{-6mm}v_{(g,1_P)}v_{(g^{-1},1_P)}&\hspace*{-2.5mm}=&\hspace*{-2.5mm} v_{(g,1_P)(g^{-1},1_P)} = v_{(1_G,1_P)} = 1 \vspace*{2mm}\\
\text{(CNP 1)}&\hspace*{-6mm} v_{(1_G,p)}v_{(g,1_P)} &\hspace*{-2.5mm}=&\hspace*{-2.5mm} v_{(1_G,p)(g,1_P)} = v_{(\theta_p(g),p)} = v_{(\theta_p(g),1_P)}v_{(1_G,p)}\vspace*{2mm}\\
\text{(CNP 2)}&\hspace*{-6mm} v_{(1_G,p)}^*v_{(g,1_P)}v_{(1_G,q)} &\hspace*{-2.5mm}\stackrel{!}{=}&\hspace*{-2.5mm} \chi_{\theta_p(G)\theta_q(G)}(g)~v_{(g_1,(p{\wedge}q)^{-1}q)}v_{(g_2^{-1},(p{\wedge}q)^{-1}p)}^*\\
&\hspace*{-6mm}&&\hspace*{-2.5mm}\text{where } g = \theta_p(g_1)\theta_q(g_2).\vspace*{2mm}\\
\Longleftrightarrow&\hspace*{-6mm} v_{(1_G,p)}v_{(1_G,p)}^*v_{(g,q)}v_{(g,q)}^* &\hspace*{-2.5mm}\stackrel{!}{=}&\hspace*{-2.5mm} \chi_{\theta_p(G)\theta_q(G)}(g)~v_{(\theta_p(g_1),p{\vee}q)}v_{(g\theta_q(g_2^{-1}),p{\vee}q)}^*\vspace*{2mm}\\
\Longleftrightarrow&\hspace*{-6mm} e_{1_G,p}e_{g,q} &\hspace*{-2.5mm}\stackrel{!}{=}&\hspace*{-2.5mm} \chi_{\theta_p(G)\theta_q(G)}(g)~e_{(g\theta_q(g_2^{-1}),p{\vee}q)}
\end{array} \]  
as $g = \theta_p(g_1)\theta_q(g_2)$ gives $\theta_p(g_1) = g\theta_q(g_2^{-1})$. This last equation holds by Lemma~\ref{lem:e_(g,p) commute}, so (CNP 2) is satisfied as well. (CNP 3) is a relation that is encoded inside $\CD$, so it is satisfied as the range projection of the isometry $v_{(g,p)}$ coincides with $e_{g,p}$. Injectivity of $\varphi$ follows from the fact that the isometries $u_gs_p$ satisfy the covariance relation for the action of $G\rtimes_\theta P$ on $\CD$ since $u_gs_p e_{h,q} (u_gs_p)^* = e_{g\theta_p(h),pq} = (g,p).e_{h,q}$. Indeed, in this case there is a surjective $*$-homomorphism from $\CD{\rtimes}(G{\rtimes_\theta}P)$ to $\CO[G,P,\theta]$ sending $v_{(g,p)}$ to $u_gs_p$ and the two $*$-homomorphisms are mutually inverse, so $\varphi$ is an isomorphism.  
\end{proof}

\noindent This description of $\CO[G,P,\theta]$ allows us to deduce several relevant properties of $\CO[G,P,\theta]$ and its core subalgebra $\CF$.

\begin{corollary}\label{cor:F cong D rtimes G-general IAD}
The isomorphism $\varphi$ from Proposition~\ref{prop:O cong D rtimes GxP} restricts to an isomorphism between $\CF$ and $\CD{\rtimes}G$. In particular, we have a canonical isomorphism $\CO[G,P,\theta] \cong \CF{\rtimes}P$.
\end{corollary}
\begin{proof}
The first claim follows immediately from Proposition~\ref{prop:O cong D rtimes GxP} together with Theorem~\ref{thm:cr prod by sd prod as it cr prod} and Remark~\ref{rem:ext to proper endomorphisms for non-unital coeff algs}. The second assertion is implied by Lemma~\ref{lem:F dense span family}.
\end{proof}

\begin{proposition}\label{prop:F nuclear and UCT}
If the $G$-action $\hat{\tau}$ on $G_\theta$ is amenable, then both $\CF$ and $\CO[G,P,\theta]$ are nuclear and satisfy the universal coefficient theorem (UCT).
\end{proposition}
\begin{proof}
As $\CF \cong \CD\rtimes G$ by Corollary~\ref{cor:F cong D rtimes G-general IAD} and $\hat{\tau}$ is amenable, $\CF$ is nuclear by results of Claire Anatharaman-Delaroche, see \cite{AD} or \cite{BO}*{Theorem 4.3.4}. Similarly, amenability of $\hat{\tau}$ passes to the corresponding transformation groupoid $\CG$. Thus, we can rely on results of Jean-Louis Tu, see \cite{Tu}, to deduce that $\CF \cong \CD \rtimes_\tau G \cong C^*(\CG)$ satisfies the UCT. The class of separable nuclear C*-algebras that satisfy the UCT is closed under crossed products by $\N$ and inductive limits. Recall that either $P \cong \N^k$ for some $k \in \N$ or $P \cong \bigoplus_\N \N$ according to condition (B) of Definition~\ref{def:IAD}. Hence the claims concerning $\CO[G,P,\theta]$ follow from $\CO[G,P,\theta] \cong \CF \rtimes P$, see Corollary~\ref{cor:F cong D rtimes G-general IAD}. 
\end{proof}

\begin{corollary}\label{cor:E_2}
The map $E_{2}(u_gs_ps_p^*u_h^*) := \delta_{g h}~e_{g,p}$ defines a conditional expectation $E_2: \CF \longrightarrow \CD$ which is faithful if and only if $\hat{\tau}$ is amenable. 
\end{corollary}
\begin{proof}
Due to Corollary~\ref{cor:F cong D rtimes G-general IAD}, $\CF$ is canonically isomorphic to $\CD \rtimes_\tau G$. Since $G$ is discrete, the reduced crossed product $\CD \rtimes_{\tau,r} G$ has a faithful conditional expectation given by evaluation at $1_G$. The map $E_2$ is nothing but the composition of 
\[\CF \cong \CD\rtimes_\tau G \onto \CD\rtimes_{\tau,r}G \stackrel{ev_{1_G}}{\longrightarrow} \CD.\]
By \cite{AD}, the canonical surjection $\CD \rtimes_\tau G \onto \CD \rtimes_{r,\tau} G$ is an isomorphism if and only if $\hat{\tau}$ is amenable. 
\end{proof}

\begin{corollary}\label{cor:cond exp O->D}
The map $E(u_gs_ps_q^*u_h^*) := \delta_{p q} \delta_{g h} e_{g,p}$ defines a conditional expectation $E: \CO[G,P,\theta] \longrightarrow \CD$ which is faithful if and only if $\hat{\tau}$ is amenable.
\end{corollary}
\begin{proof}
Clearly, $E = E_{2} \circ E_{1}$, so the result follows from Remark~\ref{rem:E_1} and Corollary~\ref{cor:E_2}. 
\end{proof}

\noindent Note that if $G$ happens to be amenable, the faithful conditional expectation $E$ can be obtained directly by showing that the left Ore semigroup $G{\rtimes_\theta}P$ has an amenable enveloping group. Before we can turn to simplicity of $\CO[G,P,\theta]$, we need the following general observations: 

\begin{definition}
Given a family of commuting projections $(E_{i})_{i \in I}$ in a unital C*-algebra $B$ and finite subsets $A \subset F$ of $I$, let
\[\begin{array}{c} Q_{F,A}^{E} := \prod\limits_{i \in A}E_{i}\prod\limits_{j \in F \setminus A}(1-E_{j}). \end{array}\]
Products indexed by $\emptyset$ are treated as $1$ by convention.
\end{definition}

\begin{lemma}\label{lem:first Qs}
Suppose $(E_{i})_{i \in I}$ is a family of commuting projections in a unital C*-algebra $B$, $A \subset F$ are finite subsets of $I$. Then each $Q_{F,A}^{E}$ is a projection, $\sum_{A \subset F}Q_{F,A}^{E} = 1$ and, for all $\lambda_{i} \in \C, i \in F$, we have
\[\begin{array}{rcl}\sum\limits_{i \in F}\lambda_{i}E_{i}=\sum\limits_{A\subset F}\Big(\sum\limits_{i\in A}\lambda_{i}\Big) Q_{F,A}^{E} &\text{ and }& \Big\|\sum\limits_{i\in F}\lambda_{i}E_{i}\Big\|=\max\limits_{\substack{A \subset F \\ Q_{F,A}^{E} \neq 0}}\big|\sum\limits_{i\in A}\lambda_{i}\big|.
\end{array}\]
\end{lemma}
\begin{proof}
Since the projections $E_{i}$ commute, $Q_{F,A}^{E}$ is a projection. The second assertion is obtained via $1 = \prod_{i \in F}{(E_{i} + 1- E_{i})}  = \sum_{A \subset F}Q_{F,A}^{E}$. The two equations from the claim follow immediately from this.
\end{proof}

\begin{lemma}\label{lem:small subprojections2}
For $d = \sum_{i=1}^n{\lambda_{i} e_{g_{i},p_{i}}} \in \CD_+$ with $\lambda_{i} \in \C$ and $(g_{i},p_{i}) \in G\times P$, there exist $(g,p) \in G\times P$ satisfying $de_{g,p} = \|d\|e_{g,p}$.   
\end{lemma}
\begin{proof}
$d$ is contained in $C^*(\{Q_{F,A}^e \mid A \subset F = \{(g_{i},p_{i}) \mid 1 \leq i \leq n\}\})$, which is commutative by Lemma~\ref{lem:e_(g,p) commute}. Then Lemma~\ref{lem:first Qs} says that there exists $A \subset F$ such that $Q_{F,A}^{e}$ is non-zero and $dQ_{F,A}^{e} = \|d\|Q_{F,A}^{e}$. In particular, $\prod_{(g,p) \in A}{e_{g,p}}$ is non-zero, so Lemma~\ref{lem:e_(g,p) commute} implies that there exist $g_A \in G$ and $p_A\in P$ such that $\prod_{(g,p) \in A}{e_{g,p}} = e_{g_A,p_A}$. Thus, we can apply Lemma~\ref{lem:small subprojections1} to $e_{g_A,p_A}\prod_{(h,q) \in F\setminus A}(1-e_{h,q}) = Q_{F,A}^{e} \neq 0$ and the proof is complete.
\end{proof}

\noindent Note that the hard part of the proof for Lemma~\ref{lem:small subprojections2} is hidden in Lemma~\ref{lem:small subprojections1}.

\begin{theorem}\label{thm:O p.i. and simple}
If $(G,P,\theta)$ is minimal and the action $\hat{\tau}$ is amenable, then $\CO[G,P,\theta]$ is purely infinite and simple.
\end{theorem}
\begin{proof}
The linear span of $(u_gs_ps_q^*u_h^*)_{g,h \in G,p,q \in P}$ is dense in $\CO[G,P,\theta]$ according to Lemma~\ref{lem:dense span}. Every element $z$ from this linear span is of the form
\[\begin{array}{c} z = \sum\limits_{i = 1}^{m_{1}}{c_i e_{g_i,p_i}} 
		+ \sum\limits_{i = m_{1}+1}^{m_{2}}{c_i u_{g_i}s_{p_i}s_{p_i}^*u_{h_i}^*} 
		+ \sum\limits_{i = m_{2}+1}^{m_{3}}{c_i u_{g_i}s_{p_i}s_{q_i}^*u_{h_i}^*}, \end{array}\]
where $c_{i} \in \C$, 
\begin{enumerate}[a)]
\item $g_i \neq h_i$ for $m_{1}+1 \leq i \leq m_{2}$, and 
\item $p_i \neq q_i$ for $m_{2}+1 \leq i \leq m_{3}$.
\end{enumerate}
By Corollary~\ref{cor:cond exp O->D}, we have $E(z) = \sum_{i = 1}^{m_{1}}{c_i e_{g_i,p_i}} \in \CD$. If we assume $z$ to be non-zero and positive, which we will do from now on, then $E(z) > 0$ as $E$ is a faithful conditional expectation. Applying Lemma~\ref{lem:small subprojections2} to $E(z)$ yields $(g,p) \in G\times P$ such that 
\begin{enumerate}
\item[c)] $E(z)e_{g,p} = \|E(z)\|e_{g,p}$.
\end{enumerate}
In order to prove simplicity and pure infiniteness of $\CO[G,P,\theta]$, it suffices to establish the following claim: There exist $(\tilde{g},\tilde{p}) \in G\times P$ satisfying
\begin{enumerate}[(a)]
\item $e_{\tilde{g},\tilde{p}} \leq e_{g,p}$,
\item $e_{\tilde{g},\tilde{p}} u_{g_i}s_{p_i}s_{p_i}^*u_{h_i}^* e_{\tilde{g},\tilde{p}} = 0$ for $m_{1}+1 \leq i \leq m_{2}$ and
\item $e_{\tilde{g},\tilde{p}} u_{g_i}s_{p_i}s_{q_i}^*u_{h_i}^* e_{\tilde{g},\tilde{p}} = 0$ for $m_{2}+1 \leq i \leq m_{3}$.
\end{enumerate}
Indeed, if this can be done, then we get 
\[e_{\tilde{g},\tilde{p}}ze_{\tilde{g},\tilde{p}} \stackrel{(b),(c)}{=} e_{\tilde{g},\tilde{p}}E(z)e_{\tilde{g},\tilde{p}} \stackrel{c),(a)}{=} \|E(z)\| e_{\tilde{g},\tilde{p}}.\]
Now for $x \in \CO[G,P,\theta]$ positive and non-zero, let $\varepsilon > 0$ and choose a positive, non-zero element $z$, which is a finite linear combination of elements $u_{g'}s_{p'}s_{q'}^*u_{h'}^*$, to approximate $x$ up to $\varepsilon$. Then $\|E(z)\|$ is a non-zero positive element of $\CD$. Thus, choosing $e_{\tilde{g},\tilde{p}}$ as above, we see that $e_{\tilde{g},\tilde{p}}ze_{\tilde{g},\tilde{p}} = \|E(z)\| e_{\tilde{g},\tilde{p}}$ is invertible in $e_{\tilde{g},\tilde{p}} \CO[G,P,\theta] e_{\tilde{g},\tilde{p}}$. If $\|x-z\|$ is sufficiently small, this implies that $e_{\tilde{g},\tilde{p}}xe_{\tilde{g},\tilde{p}}$ is positive and invertible in $e_{\tilde{g},\tilde{p}} \CO[G,P,\theta] e_{\tilde{g},\tilde{p}}$ as well because $\|E(z)\| \stackrel{\varepsilon \rightarrow 0}{\longrightarrow}~\|E(x)\| > 0$. Hence, if we denote its inverse by $y$,	then
\[\left(y^{\frac{1}{2}}u_{\tilde{g}}s_{\tilde{p}}\right)^{*}e_{\tilde{g},\tilde{p}}xe_{\tilde{g},\tilde{p}}\left(y^{\frac{1}{2}}u_{\tilde{g}}s_{\tilde{p}}\right) = 1.\vspace*{3mm}\]

\noindent We claim that there is a pair $(\tilde{g},\tilde{p}) \in G\times P$ satisfying (a)--(c). Let $(g',p') \in g\theta_{p}(G) \times pP$ and $m_{1}+1 \leq i \leq m_{2}$. Noting that $u_{g_i}s_{p_i}s_{p_i}^*u_{h_i}^* = u_{g_i h_i^{-1}}e_{h_i,p_i}$, Lemma \ref{lem:e_(g,p) commute} implies
\[\begin{array}{lcl}
e_{g',p'}u_{g_i h_i^{-1}}e_{h_i,p_i}e_{g',p'} 
&=& e_{g',p'}u_{g_i h_i^{-1}}e_{g',p'}e_{h_i,p_i}\\ 
&=& \chi_{\theta_{p'}(G)}((g')^{-1}g_i h_i^{-1}g')~u_{g_i h_i^{-1}}e_{g',p'}e_{h_i,p_i}. 
\end{array}\]
According to a), we have $(g')^{-1}g_i h_i^{-1}g' \neq 1_G$. Thus, minimality of $(G,P,\theta)$ provides $p_i' \in pP$ with the property that $(g')^{-1}g_i h_i^{-1}g' \notin \theta_{p_{i}'}(G)$. So if we take $p^{(b)} := \bigvee\limits_{i = m_{1}+1}^{m_{2}}{p_i'}$, then (a) and (b) of the claim hold for all $(g',p') \in g\theta_{p}(G)\times p^{(b)}P$. Let us assume that $p' \geq p^{(b)} \vee \bigvee_{i = m_{2}+1}^{m_{3}}{p_i \vee q_i}$ and $g' \in g \circ \theta_{p'}(G)$. Then condition (c) holds for $(g',p')$ if and only if
\[\begin{array}{lcl}
0 &\hspace*{-2.5mm}=&\hspace*{-2.5mm} s_{p'}^*u_{(g')^{-1}g_i}s_{p_i}s_{q_i}^*u_{h_i^{-1}g'}s_{p'} \\
&\hspace*{-2.5mm}=&\hspace*{-2.5mm} \chi_{\theta_{p_i}(G)}((g')^{-1}g_i)\chi_{\theta_{q_i}(G)}(h_i^{-1}g') s^*_{p_i^{-1}p'}u_{\theta_{p_i}^{-1}((g')^{-1}g_i)\theta_{q_i}^{-1}(h_i^{-1}g')}s_{q_i^{-1}p'}
\end{array}\] 	
is valid for all $m_{2}+1 \leq i \leq m_{3}$. This is precisely the case if at least one of the conditions
\begin{enumerate}[$\cdot$]
\item $(g')^{-1}g_i \in \theta_{p_i}(G)$,
\item $(g')^{-1}h_i \in \theta_{q_i}(G)$, or
\item $\theta_{p_i}^{-1}((g')^{-1}g_i)\theta_{q_i}^{-1}(h_i^{-1}g') \in \theta_{(p_i \vee q_i)^{-1}p'}(G)$
\end{enumerate}
fails for each $i$. Suppose, we have an index $i$ for which the first two conditions are satisfied. Using injectivity of $\theta_{p_i \vee q_i}$, the third condition is equivalent to $\theta_{r_q}((g')^{-1}g_i) \theta_{r_p}(h_i^{-1}g') \in \theta_{p'}(G)$, where $r_p:= (p_i \wedge q_i)^{-1}p_i$ and $r_q:= (p_i \wedge q_i)^{-1}q_i$. Condition b) implies $r_p \wedge r_q = 1_P \neq r_p r_q$. Moreover, we have 
\[\theta_{r_q}((g')^{-1}g_i) \theta_{r_p}(h_i^{-1}g') = 1_G \Longleftrightarrow 
\theta_{r_q}(g')\theta_{r_p}(g')^{-1} = \theta_{r_q}(g_i) \theta_{r_p}(h_i^{-1}).\] 
Let us examine the range of the map $f_i: G \longrightarrow G$ that is defined by $g \mapsto \theta_{r_q}(g)\theta_{r_p}(g)^{-1}$. Note that $f_{i}$ need not be a group homomorphism unless $G$ is abelian, in which case the following part can be shortened. If $k_1,k_2 \in G$ have the same image under $f_i$, then $\theta_{r_p}(k_2^{-1}k_1) = \theta_{r_q}(k_2^{-1}k_1)$. By (C1) from Definition~\ref{def:IAD}, this gives 
$k_2^{-1}k_1 \in \theta_{r_p}(G) \cap \theta_{r_q}(G) = \theta_{r_p r_q}(G)$.
But if $k_2^{-1}k_1 = \theta_{r_p r_q}(k_3)$ holds for some $k_3 \in G$, then $\theta_{r_p}(k_2^{-1}k_1) = \theta_{r_q}(k_{2}^{-1}k_1)$ implies that $\theta_{r_p}(k_3) = \theta_{r_q}(k_3)$ holds as well because $P$ is commutative and $\theta_{q_{i,1}q_{i,2}}$ is injective. By induction, we get $k_2^{-1}k_1 \in \bigcap_{n \in \N}{\theta_{(r_p r_q)^n}(G)}$.

Hence $f_i^{-1}(\theta_{r_p}(h_i)\theta_{r_q}(g_i^{-1}))$ is either empty, in which case there is nothing to do, or it is of the form $\tilde{g}_i~\bigcap_{n \in \N}{\theta_{(r_p r_q)^n}(G)}$ for a suitable $\tilde{g}_i \in G$. But for the collection of those $i$ for which the preimage in question is non-empty, we can apply Lemma~\ref{lem:believe it or not} to obtain $\tilde{g} \in g\theta_{p'}(G)$ such that $f_{i}(\tilde{g}) \neq \theta_{r_p}(h_i) \theta_{r_q}(g_i^{-1})$ for all relevant $i$.

By condition (C2) from Definition~\ref{def:IAD}, we can choose  $\tilde{p} \geq p'$ large enough so that these elements are still different modulo $\theta_{(p_i \vee q_i)^{-1}\tilde{p}}(G)$ for all $i$. In this case, we get
\[\theta_{p_i}^{-1}(\tilde{g}^{-1}g_i)\theta_{q_i}(h_i^{-1}\tilde{g}) \notin \theta_{(p_i \vee q_i)^{-1}\tilde{p}}(G) \text{ for all } m_2+1 \leq i \leq m_3,\]
so $(\tilde{g},\tilde{p})$ satisfies (c). In other words, we have proven that the pair $(\tilde{g},\tilde{p})$ satisfies (a)--(c). Thus, $\CO[G,P,\theta]$ is purely infinite and simple. 
\end{proof}

\noindent From this result, we easily get the following corollaries:

\begin{corollary}
If $(G,P,\theta)$ is minimal and $\hat{\tau}$ is amenable, then the representation $\lambda: \CO[G,P,\theta] \longrightarrow \CL(\ell^2(G))$ from Proposition~\ref{prop:O-canonical rep} is faithful. 	
\end{corollary}
\begin{proof}
This follows readily from Proposition~\ref{prop:O-canonical rep} and simplicity of $\CO[G,P,\theta]$.	
\end{proof}

\noindent Combining Lemma~\ref{lem:dense span}, Theorem~\ref{thm:O p.i. and simple} and Proposition~\ref{prop:F nuclear and UCT}, we get:

\begin{corollary}\label{cor:UCT Kirchberg algebra}
If $(G,P,\theta)$ is minimal and $\hat{\tau}$ is amenable, then $\CO[G,P,\theta]$ is a unital UCT Kirchberg algebra.
\end{corollary}

\noindent Thus, minimal irreversible algebraic dynamical systems $(G,P,\theta)$ for which the action $\hat{\tau}$ is amenable yield C*-algebras $\CO[G,P,\theta]$ that are classified by their K-theory, see \cites{Kir,Phi}. Let us come back to some of the examples from Section~1.1 and briefly describe the structure of the C*-algebras obtained in the various cases:

\begin{examples}\label{ex:CA IAD standard}~
\begin{enumerate}[(a)]
\item Let $G = \Z$, $(p_i)_{i \in I} \subset \Z{\setminus}\{0,\pm 1\}$ be a family of relatively prime integers, and set $P = |(p_i)_{i \in I}\rangle \subset \Z^\times$, which acts on $G$ by $\theta_i(g) = p_ig$. We know from the considerations in Example~\ref{ex:IAD standard}~(a) that $(G,P,\theta)$ is  minimal, so $\CO[G,P,\theta]$ is a unital UCT Kirchberg algebra. If we denote $p:= \prod_{i \in I}|p_i| \in \N \cup \{\infty\}$, then $G_\theta$ can be identified with the $p$-adic completion $\Z_p = \varprojlim (\Z/q\Z,\theta_q)_{q \in P}$ of $\Z$. Moreover, $\CF$ is the Bunce-Deddens algebra of type $p^\infty$, see \cite{BD} for the classification of Bunce-Deddens algebras by supernatural numbers. 
\item Let $I\subset\N$, choose $\{q\} \cup (p_i)_{i \in I} \subset \Z\setminus\Z^*$ relatively prime, $P = |(p_i)_{i \in I}\rangle$, set $G = \Z[1/q]$, and let $\theta_p(g) = pg$ for $g \in G,p \in P$. As in (a), $\CO[G,P,\theta]$ is a UCT Kirchberg algebra by the considerations in Example~\ref{ex:IAD standard}~(b) and Corollary~\ref{cor:UCT Kirchberg algebra}. If $p:= \prod_{i \in I}|p_i| \in \N \cup \{\infty\}$, then $G_\theta$ can be thought of as a $p$-adic completion of $\Z[1/q]$ and we obtain $\CF \cong C(G_\theta) \rtimes_\tau \Z[1/q]$.
\end{enumerate}
\end{examples}

\begin{example}\label{ex:CA IAD for On}
We have seen in Example~\ref{ex:IAD for On} that for $n \geq 2$, the dynamical system given by the unilateral shift on $G = \bigoplus_\N \Z/n\Z$ is a minimal commutative irreversible algebraic dynamical system of finite type. It has been observed in \cite{CV} that $\CO[G,P,\theta]$ is isomorphic to $\CO_n$ in a canonical way: If $e_1 = (1,0,0,\dots,) \in G$, $s \in \CO[G,P,\theta]$ denotes the generating isometry for $P$ and $s_1,\dots,s_n$ are the generating isometries of $\CO_n$, then this isomorphism is given by $u_{ke_1}s \mapsto s_k$ for $k=1,\dots,n$. In particular, $\CF$ is the UHF algebra of type $n^\infty$ and $G_\theta$ is homeomorphic to the space of infinite words using the alphabet $\{1,\dots,n\}$.
\end{example}


\begin{example}\label{ex:CA IAD as sum of IADs}
Given a family $(G^{(i)},P,\theta^{(i)})_{i\in \N}$, where each $(G^{(i)},P,\theta^{(i)})$ is an irreversible algebraic dynamical system, we can consider $G := \bigoplus_{i \in \N} G^{(i)}$, on which $P$ acts component-wise. Assume that each $(G^{(i)},P,\theta^{(i)})$ and hence $(G,P,\theta)$ is minimal, compare Example~\ref{ex:IAD as sum of IADs}. We have $G_\theta \cong \prod_{i \in I} G^{(i)}_{\theta^{(i)}}$. Thus the $G$-action $\hat{\tau}$ on $ G_\theta$ is amenable if and only if each $G_i$-action $\hat{\tau}_i$ on $ G^{(i)}_{\theta^{(i)}}$ is amenable. As $G$ is commutative (amenable) if and only if each $G^{(i)}$ is, there are various cases where amenability of $\hat{\tau}$ is for granted. In such situations, $\CO[G,P,\theta]$ is a unital UCT Kirchberg algebra.
\end{example}

\begin{example}\label{ex:CA IAD free group}
For the examples arising from free group $\IF_n$ with $2 \leq n \leq \infty$, see Example~\ref{ex:IAD free group}, we are able to provide criteria (1)--(3) to ensure that we obtain minimal irreversible algebraic dynamical systems. Hence $G_\theta$ can be interpreted as a certain completion of $\IF_n$ with respect to $\theta$. Now $\IF_n$ is far from being amenable, but the action $\hat{\tau}$ could still be amenable: The free groups are known to be exact. By a famous result of Narutaka Ozawa, exactness of a discrete group is equivalent to amenability of the left translation action on its Stone-\v{C}ech compactification, see \cite{Oza}. Recently, Mehrdad Kalantar and Matthew Kennedy have shown that exactness of a discrete group is also determined completely by amenability of the natural action on its Furstenberg boundary, see \cite{KK} for details. The latter space is usually substantially smaller than the Stone-\v{C}ech compactification and their methods may give some insights into the question of amenability in the context of the examples presented here.
\end{example}

\section{The finite type case revisited}\label{sec4}
\noindent This section provides a more detailed presentation of the case where $(G,P,\theta)$ is of finite type. In particular, we exhibit additional structural properties of the spectrum $G_\theta$ of the diagonal $\CD$ in $\CO[G,P,\theta]$. For instance, the assumption that $\theta_p(G) \subset G$ is normal for every $p \in P$ causes $G_\theta$ to inherit the group structure from $G$. This turns $G_\theta$ into a profinite group. 
If, in addition, $(G,P,\theta)$ is minimal and $G$ is amenable, then $\CF$ falls into the class of generalised Bunce-Deddens algebras, see \cites{Orf,Car} for details. Due to \cites{Lin,MS,Win}, they belong to a large class of C*-algebras that can be classified by K-theory. 

We are particularly interested in the case where $G$ is abelian. For such dynamical systems, the situation is significantly easier as $\theta_p(G) \subset G$ is normal for all $p \in P$ and the action $\hat{\tau}$ is always amenable. In fact, the structure of $\CD$ and $\CF$ is quite similar to the one discovered in the singly generated case, compare \cite{CV}*{Section 2}: $G_\theta$ is a compact abelian group and we have a chain of isomorphisms $\CF \cong C(G_\theta) \rtimes_\tau G \cong C(\hat{G}) \rtimes_{\bar{\tau}} \hat{G}_\theta$. Throughout this section, we will assume that $(G,P,\theta)$ is an irreversible algebraic dynamical system of finite type.

\begin{remark}\label{rem:CA IADoFT basic}
Recall from Remark~\ref{rem:IAD SpecD as completion} that $G_\theta$ can be thought of as a completion of $G$ with respect to $\theta$ provided that $(G,P,\theta)$ is minimal. The map $\iota$ from Lemma~\ref{lem:IAD spec D as a completion of G} transports more structure under additional hypotheses: 
\begin{enumerate}[a)]
\item If $\theta_p(G)$ is normal in $G$ for all $p \in P$, then $G_{\theta}$ is a profinite group. 
\item If $(G,P,\theta)$ is minimal and $\theta_p(G)$ is normal in $G$ for all $p \in P$, then $\iota$ is a dense embedding of groups. In particular, the $G$-action $\hat{\tau}$ on $G_\theta$ is the left translation action of a dense subgroup in $G_\theta$.
\item $G_{\theta}$ is an abelian group if and only if $G$ is an abelian group. So if $(G,P,\theta)$ is a minimal commutative irreversible algebraic dynamical system, then $G_\theta$ is a compact abelian group and $\CF$ has a unique tracial state by b). This follows from a straightforward adaptation of the corresponding part of the proof for \cite{CV}*{Lemma 2.5}   
\end{enumerate}
\end{remark}



\begin{proposition}\label{prop:gen BD-alg for IADoFT}
Suppose $(G,P,\theta)$ is minimal and $G$ is amenable. Then $\CF$ is a generalised Bunce-Deddens algebra.
\end{proposition}
\begin{proof}
This follows directly from the construction of the generalised Bunce-Deddens algebras presented in \cite{Orf}*{Section 2}: Choose an arbitrary, increasing, cofinal sequence $(p_n)_{n \in \N} \subset P$, where cofinal means that, for every $q \in P$, there exists an $n \in \N$ such that $p_n \in qP$. Then $(\theta_{p_n}(G))_{n \in \N}$ is a family of nested, normal subgroups of finite index in $G$. This family is separating for $G$ by minimality of $(G,P,\theta)$.
\end{proof}

\noindent In particular, these assumptions force $\CF$ to be unital, nuclear, separable, simple, quasidiagonal, and to have real rank zero, stable rank one, strict comparison for projections as well as a unique tracical state, see \cite{Orf}. As the combination of real rank zero and strict comparison for projections yields strict comparison (for positive elements), the prerequisites for \cite{MS}*{Theorem 1.1} are met, so $\CF$ also has finite decomposition rank. This establishes the remaining step to achieve classification of the core $\CF$ by means of its Elliott invariant $\left(K_0(\CF),K_0(\CF)_+,[1_{\CF}],K_1(\CF)\right)$ thanks to results of Huaxin Lin and Wilhelm Winter, see \cite{Win}*{Corollary 6.5(i)} and \cite{Lin}:

\begin{corollary}\label{cor:IAD class of the core}
Let $(G_i,P_i,\theta_i)$ be minimal and $G_i$ be amenable for $i = 1,2$. If $\CF_1$ and $\CF_2$ denote the respective cores, then $\CF_1 \cong \CF_2$ holds if and only if
\[\left(K_0(\CF_1),K_0(\CF_1)_+,[1_{\CF_1}],K_1(\CF_1)\right) \cong \left(K_0(\CF_2),K_0(\CF_2)_+,[1_{\CF_2}],K_1(\CF_2)\right).\]

\end{corollary}

\noindent We close this section by presenting an intriguing isomorphism of group crossed products on the level of $\CF$, compare \cite{CV}*{Lemma 2.5}:

\begin{corollary}\label{cor:dual cr prod for aIADoFT}
Let $(G,P,\theta)$ be commutative and minimal. Then there is a $\hat{G}_\theta$-action $\bar{\tau}$ on $C(\hat{G})$ for which $\CF \cong C(G_\theta) \rtimes_\tau G \cong C(\hat{G}) \rtimes_{\bar{\tau}} \hat{G}_\theta$.
\end{corollary}
\begin{proof}
The first isomorphism has been achieved in Corollary~\ref{cor:F cong D rtimes G-general IAD}. For the second part, let $\bar{\tau}_{\chi_\theta} (\chi)(g) := \chi_\theta(\iota(g))\chi(g)$ for $\chi_\theta \in \hat{G}_\theta, \chi \in \hat{G}$ and $g \in G$. Since $\iota: G \longrightarrow G_\theta$ is a group homomorphism, $\bar{\tau}_{\chi_\theta} (\chi)$ defines a character of $G$. Clearly, $\hat{\tau}$ is compatible with the group structure on $\hat{G}_\theta$. According to Remark~\ref{rem:CA IADoFT basic}~b) the group homomorphism $\iota$ identifies $G$ with a dense subgroup of $G_\theta$. In this case the characters on $G_\theta$ are in one-to-one correspondence with the characters on $G$. Note that this correspondence is precisely given by regarding characters on $G_\theta$ as characters on $G$ using $\iota$. Therefore, $\bar{\tau}$ defines an action of $\hat{G}_\theta$ by homeomorphisms of the compact space $\hat{G}$. Once we know that $\bar{\tau}$ defines an action, we readily see that there is a canonical surjective $*$-homomorphism $C(G_\theta) \rtimes_\tau G \onto C(\hat{G}) \rtimes_{\bar{\tau}} \hat{G}_\theta$. As $C(G_\theta) \rtimes_\tau G$ is simple, this map is an isomorphism.  
\end{proof}

\section{A product systems perspective}\label{sec5}
\noindent This section is designed to provide a product system of Hilbert bimodules for each irreversible algebraic dynamical system $(G,P,\theta)$. The features of $(G,P,\theta)$ result in a particularly well-behaved product system $\CX$. Therefore, it is possible to obtain a concrete presentation of $\CO_{\CX}$ from the data of the dynamical system. In the case of irreversible algebraic dynamical systems of finite type, this algebra is shown to be isomorphic to $\CO[G,P,\theta]$. 

The corresponding result in the general case, that is, allowing for the presence of group endomorphisms $\theta_p$ of $G$ with infinite index, requires a more involved argument. The reason is that the prerequisites for \cite{Sta2}*{Corollary 3.21} are not met, so one has to deal with Nica covariance of representations. Since this is more closely related to the Nica-Toeplitz algebra $\CN\CT_\CX$, we will only treat the finite type case and refer to \cite{BLS2} for the strategy in the general case. In fact, the proof in  \cite{BLS2} reveals a close connection between Nica covariance and (CNP 2), relying essentially on independence of group endomorphisms for relatively prime elements of $P$. More precisely, it shows that, for $\CX$ associated to $(G,P,\theta)$, Nica covariance boils down to its original form, see \cite{Nic}: A representation $\varphi$ of the product system $\CX$ is Nica covariant if and only if $\varphi_p(1_{C^*(G)})$ and $\varphi_q(1_{C^*(G)})$ are doubly commuting isometries whenever $p$ and $q$ are relatively prime in $P$.

We start with a brief recapitulation of the necessary definitions for product systems and Cuntz-Nica-Pimsner covariance, compare \cite{Sta2}*{Section 3}.

\begin{definition}\label{def:prod system}
A \textit{product system of Hilbert bimodules} over a monoid $P$ with coefficients in a C*-algebra $A$ is a monoid $\CX$ together with a monoidal homomorphism $\rho:\CX \longrightarrow P$ such that:
\begin{enumerate}[(1)]
\item $\CX_{p} := \rho^{-1}(p)$ is a Hilbert bimodule over $A$ for each $p \in P$,
\item $\CX_{1_P} \cong \hspace*{1mm}_{\id}\hspace*{-0.5mm}A_{\id}$ as Hilbert bimodules and 
\item for all $p,q \in P$, we have $\CX_{p} \otimes_{A} \CX_{q} \cong \CX_{pq}$ if $p \neq 1_P$, and $\CX_{1_P} \otimes_{A} \CX_{q} \cong \overline{\phi_{q}(A)\CX_{q}}$.
\end{enumerate}   
\end{definition}

\begin{definition}\label{def:ONB HB}
Let $\CH$ be a Hilbert bimodule over a C*-algebra $A$ and $(\xi_{i})_{i \in I} \subset \CH$. Consider the following properties:
\[\begin{array}{llll} (1)\hspace*{-2mm}& \left\langle \xi_{i},\xi_{j} \right\rangle = \delta_{ij}1_{A} \text{ for all } i,j \in I. \hspace*{2mm}& (2)\hspace*{-2mm} & \eta = \sum\limits_{i \in I}{\xi_{i}\left\langle \xi_{i},\eta \right\rangle} \text{ for all } \eta \in \CH. \end{array}\]
If $(\xi_{i})_{i \in I}$ satisfies (1) and (2), it is called an \textit{orthonormal basis} for $\CH$.
\end{definition}

\begin{lemma}\label{lem:ONB gives matrix units}
Let $\CH$ be a Hilbert bimodule. If $(\xi_{i})_{i \in I} \subset \CH$ is an orthonormal basis, then $\left(\Theta_{\xi_{i},\xi_{j}}\right)_{i,j \in I}$ is a system of matrix units and $\sum_{i \in I}{\Theta_{\xi_{i},\xi_{i}}} = 1_{\CL(\CH)}$. If $\CH$ admits a finite orthonormal basis, then $\CK(\CH) = \CL(\CH)$.
\end{lemma}

\noindent This lemma is a reformulation of \cite{Sta2}*{Lemma 3.9} implies that product systems whose fibres have finite orthonormal bases are compactly aligned, see \cite{Sta2}*{Remark 3.13}. An explicit proof of this fact is presented in \cite{BLS2}.

\begin{definition}\label{def:rep of prod sys}
Let $\CX$ be a product system over $P$ and suppose $B$ is a C*-algebra. A map $\varphi: \CX \longrightarrow B$, whose fibre maps $\CX_{p} \longrightarrow B$ are denoted by $\varphi_{p}$, is called a \textit{Toeplitz representation} of $\CX$, if:
\begin{enumerate}[(1)]
\item $\varphi_{1_P}$ is a $\ast$-homomorphism.
\item $\varphi_{p}$ is linear for all $p \in P$.
\item $\varphi_{p}(\xi)^{*}\varphi_{p}(\eta) = \varphi_{1_P}\left(\left\langle \xi,\eta\right\rangle\right)$ for all $p \in P$ and $\xi,\eta \in \CX_{p}$.
\item $\varphi_{p}(\xi)\varphi_{q}(\eta) = \varphi_{pq}(\xi \eta)$ for all $p,q \in P$ and $\xi \in \CX_{p}, \eta \in \CX_{q}$.
\end{enumerate}   	
\end{definition}

\noindent A Toeplitz representation will be called a representation whenever there is no ambiguity. Given a representation $\varphi$ of $\CX$ in $B$, it induces $\ast$-homomorphisms $\psi_{\varphi,p}:\CK(\CX_{p}) \longrightarrow B$ for $p \in P$ characterised by $\Theta_{\xi,\eta} \mapsto \varphi_p(\xi)\varphi_p(\eta)^*$. If $\CX$ is compactly aligned, the representation $\varphi$ is said to be \emph{Nica covariant}, if 
$\psi_{\varphi,p}(k_{p})\psi_{\varphi,q}(k_{q}) = \psi_{\varphi,p \vee q}\left(\iota_{p}^{p \vee q}(k_{p})\iota_{q}^{p \vee q}(k_{q})\right)$
holds for all $p,q \in P$ and $k_{p} \in \CK(\CX_{p}),k_{q} \in \CK(\CX_{q})$. Concerning the choice of an appropriate notion of Cuntz-Pimsner covariance for product systems, there have been multiple attempts:

\begin{definition}\label{def:CNP cov}
Let $B$ be a C*-algebra and suppose $\CX$ is a compactly aligned product system of Hilbert bimodules over $P$ with coefficients in $A$.
\[\begin{array}{ll}
(\text{CP}_{F}) &\hspace*{-2mm} \text{A representation $\varphi: \CX \longrightarrow B$ is called \textit{Cuntz-Pimsner covariant}}\\ 
&\hspace*{-2mm}\text{in the sense of \cite{Fow2}*{Section 1}, if it satisfies}\vspace*{2mm}\\ 
&\hspace*{-2mm}\psi_{\varphi,p}(\phi_{p}(a)) = \varphi_{1_P}(a) \text{ for all } p \in P \text{ and } a \in \phi_{p}^{-1}(\CK(\CX_{p})) \subset A.\vspace*{2mm}\\
(\text{CP}) &\hspace*{-2mm} \text{A representation $\varphi: \CX \longrightarrow B$ is called \textit{Cuntz-Pimsner covariant}}\\
&\hspace*{-2mm}\text{in the sense of \cite{SY}*{Definition 3.9}, if the following holds:}\\
&\hspace*{-2mm}\text{Suppose $F \subset P$ is finite and we fix $k_{p} \in \CK(\CX_{p})$ for each $p \in F$.}\\
&\hspace*{-2mm}\text{If, for every $r \in P$, there is $s \geq r$ such that}\vspace*{2mm}\\
&\hspace*{-2mm}\begin{array}{llcll}
&\sum\limits_{p \in F}{\iota_{p}^{t}(k_{p})} &=& 0 \text{ holds for all $t \geq s$,}\vspace*{2mm}\\
\text{then }&\sum\limits_{p \in F}{\psi_{\varphi,p}(k_{p})} &=& 0 \text{ holds true.} 
\end{array}\end{array}\] 
\[\begin{array}{ll} 
(\text{CNP}) &\hspace*{-2mm} \text{A representation $\varphi: \CX \longrightarrow B$ is said to be \textit{Cuntz-Nica-Pimsner}}\\
&\hspace*{-2mm}\text{\textit{covariant}, if it is Nica covariant and $(\text{CP})$-covariant.}
\end{array}\]
\end{definition}
   
\noindent Fortunately, it was observed in \cite{SY}*{Proposition 5.1} that the different notions are closely related and in \cite{Fow2}*{Proposition 5.4}  that $(\text{CP}_{F})$ implies Nica covariance in the cases of interest to us, see \cite{Sta2}*{Subsection 3.2}.

\begin{proposition}\label{prop:PS for an IAD}
Suppose $(G,P,\theta)$ is an irreversible algebraic dynamical system. Let $(u_g)_{g \in G}$ denote the standard unitaries generating $C^*(G)$ and $\alpha$ be the action of $P$ on $C^*(G)$ induced by $\theta$, i.e. $\alpha_p(u_g) = u_{\theta_p(g)}$ for $p \in P$ and $g \in G$. Then $\CX_p := C^*(G)_{\alpha_p}$, with left action $\phi_p$ given by multiplication in $C^*(G)$ and inner product $\langle u_g,u_h \rangle_p = \chi_{\theta_p(G)}(g^{-1}h) u_{\theta_p^{-1}(g^{-1}h)}$ is an essential Hilbert bimodule. The union of all $\CX_p$ forms a product system $\CX$ over $P$ with coefficients in $C^*(G)$. $\CX$ is a product system with orthonormal bases. It is of finite type if $(G,P,\theta)$ is of finite type. 
\end{proposition} 
\begin{proof}
It is straightforward to show that $\CX$ defines a product system of essential Hilbert bimodules and we omit the details. For $p \in P$, we claim that every complete set of representatives $(g_i)_{i \in I}$ for $G/\theta_p(G)$ gives rise to an orthonormal basis of $\CX_p$. Indeed, if we fix such a transversal $(g_i)_{i \in I}$ and pick $g \in G$, then 
$\langle u_{g_i},u_g \rangle_p = \chi_{\theta_p(G)}(g_i^{-1}g) u_{\theta_p^{-1}(g_i^{-1}g)}$ equals $0$ for all but one $j \in I$, namely the one representing the left-coset $[g]$ in $G/\theta_p(G)$. Thus, the family $(u_{g_i})_{i \in I} \subset \CX_p$ consists of orthonormal elements with respect to $\langle \cdot,\cdot \rangle_p$, and $u_{g_i}\alpha_p\left(\langle u_{g_i},u_g \rangle\right) = \delta_{i j} u_g$, so $(u_{g_i})_{i \in I}$ satisfies \ref{def:ONB HB}~(2).
\end{proof}

\begin{remark}\label{rem:dichotomy finite - infinity type cp op}
If $(G,P,\theta)$ is an irreversible algebraic dynamical system and $\CX$ denotes the associated product system from Proposition~\ref{prop:PS for an IAD}, then we have already seen in the proof of Proposition~\ref{prop:PS for an IAD} that $\CX_p$ has a finite orthonormal basis if $[G:\theta_p(G)]$ is finite. Since the left action is given by left multiplication, in other words, the elements of $C^*(G)$ act as diagonal operators, we have 
\[\phi_p^{-1}(\CK(\CX_p)) = \begin{cases} C^*(G),& \text{if $[G:\theta_p(G)]$ is finite, }\\ 0,& \text{else.}\end{cases}\]
\end{remark}

\begin{lemma}\label{lem:rank-one proj from ONB - PS for IAD}
Suppose $(G,P,\theta)$ is an irreversible algebraic dynamical system and $\CX$ denotes the associated product system from Proposition~\ref{prop:PS for an IAD}. Then the rank-one projection $\Theta_{u_g,u_g} \in \CK(\CX_p)$ depends only on the equivalence class of $g$ in $G/\theta_p(G)$. Moreover, if $\varphi$ is a Nica covariant representation of $\CX$, then 
\[\begin{array}{l}
\psi_{\varphi,p}(\Theta_{u_{g_1},u_{g_1}})\psi_{\varphi,q}(\Theta_{u_{g_2},u_{g_2}})\vspace*{2mm}\\
\hspace*{4mm}= \begin{cases} \psi_{\varphi,p \vee q}(\Theta_{u_{g_3},u_{g_3}}) &\text{if $g_1^{-1}g_2 = \theta_p(g_3)\theta_q(g_4)$ for some $g_3,g_4 \in G$,}\\ 0& \text{else.}\end{cases}
\end{array}\]
holds for all $g_1,g_2 \in G$ and $p,q \in P$.
\end{lemma}
\begin{proof}
If $g_1 = g\theta_p(g_2)$ for some $g_2 \in G$, then 
\[\Theta_{u_{g_1},u_{g_1}}(u_h) = \chi_{\theta_p(G)}(\theta_p(g_2^{-1})g^{-1}h) u_{h} = \chi_{\theta_p(G)}(g^{-1}h) u_{h} = \Theta_{u_g,u_g}(u_h)\] 
for all $h \in G$ and hence $\Theta_{u_{g_1},u_{g_1}} = \Theta_{u_g,u_g}$. For the second claim, Nica covariance of $\iota_{\CO_\CX}$ implies  
\[\psi_{\varphi,p}(\Theta_{u_{g_1},u_{g_1}})\psi_{\varphi,q}(\Theta_{u_{g_2},u_{g_2}}) = \psi_{\varphi,p}(\iota_p^{p \vee q}(\Theta_{u_{g_1},u_{g_1}})\iota_q^{p \vee q}(\Theta_{u_{g_2},u_{g_2}})).\]
If we denote $p':= (p \wedge q)^{-1}p$ and $q':= (p \wedge q)^{-1}q$, then 
\[\begin{array}{c} \iota_p^{p \vee q}(\Theta_{u_{g_1},u_{g_1}}) = \sum\limits_{[g_3] \in G/\theta_{q'}(G)} \Theta_{u_{g_1\theta_p(g_3)},u_{g_1\theta_p(g_3)}} \in \CL(\CX_{p \vee q}) \end{array}\]
and
\[\begin{array}{c} \iota_q^{p \vee q}(\Theta_{u_{g_2},u_{g_2}}) = \sum\limits_{[g_4] \in G/\theta_{p'}(G)} \Theta_{u_{g_2\theta_q(g_4)},u_{g_2\theta_q(g_4)}} \in \CL(\CX_{p \vee q}) \end{array}\]
hold. We observe that  
\[\Theta_{u_{g_1\theta_p(g_3)},u_{g_1\theta_p(g_3)}}\Theta_{u_{g_2\theta_q(g_4)},u_{g_2\theta_q(g_4)}}\]
is non-zero if and only if $[g_1\theta_p(g_3)] = [g_2\theta_q(g_4)] \in G/\theta_{p \vee q}(G)$. In particular, this is always zero if $g_1^{-1}g_2 \notin \theta_p(G)\theta_q(G)$. Let us assume that there are $g_3,\dots,g_8 \in G$ such that
\[\begin{array}{lcl}
\theta_p(g_3^{-1})g_1^{-1}g_2\theta_q(g_4) &=& \theta_{p \vee q}(g_7)\\
&\text{and}\\
\theta_p(g_5^{-1})g_1^{-1}g_2\theta_q(g_6) &=& \theta_{p \vee q}(g_8).
\end{array}\] 
Rearranging the first equation to insert it into the second, we get 
\[\theta_p(g_5^{-1}g_3)\theta_{p \vee q}(g_7)\theta_q(g_4^{-1}g_6) = \theta_{p \vee q}(g_8).\]
By injectivity of $\theta_{p \wedge q}$ this is equivalent to 
\[\theta_{p'}(g_5^{-1}g_3)\theta_{(p \wedge q)^{-1}(p \vee q)}(g_7)\theta_{q'}(g_4^{-1}g_6) = \theta_{(p \wedge q)^{-1}(p \vee q)}(g_8).\]
From this equation we can easily deduce $g_5^{-1}g_3 \in \theta_{q'}(G)$ and $g_4^{-1}g_6 \in \theta_{p'}(G)$ from independence of $\theta_{p'}$ and $\theta_{q'}$, see Definition~\ref{def:IAD}~(C). Thus, if there are $g_3,g_4 \in G$ such that $\theta_p(g_3^{-1})g_1^{-1}g_2\theta_q(g_4) \in \theta_{p \vee q}(G)$, then they are unique up to $\theta_{q'}(G)$ and $\theta_{p'}(G)$, respectively. This completes the proof.
\end{proof}

\begin{theorem}\label{thm:isom ad-hoc PS for IADoFT}
Let $(G,P,\theta)$ be an irreversible algebraic dynamical system of finite type and $\CX$ the product system from Proposition~\ref{prop:PS for an IAD}. Then $u_gs_p \mapsto  \iota_{\CO_{\CX},p}(u_g)$ defines an isomorphisms $\varphi: \CO[G,P,\theta] \longrightarrow \CO_{\CX}$.
\end{theorem}
\begin{proof}
The idea is to exploit the respective universal property on both sides. We begin by showing that $(\iota_{\CO_{\CX},1_P}(u_g))_{g \in G}$ is a unitary representation of $G$ and $(\iota_{\CO_{\CX},p}(1_{C^*(G)}))_{p \in P}$ is a representation of the monoid $P$ by isometries satisfying (CNP 1)--(CNP 3), compare Definition~\ref{def:O-algebra ad-hoc}. $\iota_{\CO_{\CX},1_P}$ is a $\ast$-homomorphism, so we get a unitary representation of $G$. In addition, 
\[\begin{array}{lcl}
\iota_{\CO_{\CX},p}(1_{C^*(G)})^*\iota_{\CO_{\CX},p}(1_{C^*(G)}) &=& \iota_{\CO_{\CX},1_P}(\langle 1_{C^*(G)},1_{C^*(G)} \rangle_p)\vspace*{2mm}\\ 
&=& \iota_{\CO_{\CX},1_P}(1_{C^*(G)}) = 1_{\CO_{\CX}}
\end{array}\]
and 
\[\iota_{\CO_{\CX},p}(1_{C^*(G)})\iota_{\CO_{\CX},q}(1_{C^*(G)}) = \iota_{\CO_{\CX},pq}(1_{C^*(G)}\alpha_p(1_{C^*(G)})) = \iota_{\CO_{\CX},pq}(1_{C^*(G)})\]
show that we have a representation of $P$ by isometries. (CNP 1) follows from 
\[\iota_{\CO_{\CX},p}(1_{C^*(G)})\iota_{\CO_{\CX},1_P}(u_g) = \iota_{\CO_{\CX},p}(u_{\theta_p(g)}) = \iota_{\CO_{\CX},1_P}(u_{\theta_p(g)})\iota_{\CO_{\CX},p}(1_{C^*(G)}).\]
Let $p,q \in P$ and $g \in G$. Then (CNP 2) follows easily from applying Lemma~\ref{lem:rank-one proj from ONB - PS for IAD} to
\[\begin{array}{l}
\iota_{\CO_{\CX},p}(1_{C^*(G)})^*\iota_{\CO_{\CX},1_P}(u_g)\iota_{\CO_{\CX},q}(1_{C^*(G)}) \vspace*{2mm}\\ \hspace*{4mm}= \iota_{\CO_{\CX},p}(1_{C^*(G)})^*\psi_{\iota_{\CO_{\CX}},p}(\Theta_{1,1})\psi_{\iota_{\CO_{\CX}},q}(\Theta_{u_g,u_g})\iota_{\CO_{\CX},q}(u_g).
\end{array}\] 
Finally, we observe that 
\[\iota_{\CO_{\CX},1_P}(u_g)\iota_{\CO_{\CX},p}(1_{C^*(G)})\iota_{\CO_{\CX},p}(1_{C^*(G)})^*\iota_{\CO_{\CX},1_P}(u_g)^* = \psi_{\iota_{\CO_{\CX}},p}(\Theta_{u_g,u_g})\]
and the computation
\[\begin{array}{lclcl}
\sum\limits_{[g] \in G/\theta_p(G)}\psi_{\iota_{\CO_{\CX}},p}(\Theta_{u_g,u_g}) &=& \psi_{\iota_{\CO_{\CX}},p}(1_{\CL(\CX)}) &=& \psi_{\iota_{\CO_{\CX}},p}(\phi_p(1_{C^*(G)}))\vspace*{2mm}\\
&=& \iota_{\CO_{\CX},1_P}(1_{C^*(G)}) &=& 1_{\CO_{\CX}}
\end{array}\]
yield (CNP 3). Thus we conclude that $\varphi:\CO[G,P,\theta] \longrightarrow \CO_\CX$ defines a surjective $\ast$-homomorphism. For the reverse direction, we show that 
\[\begin{array}{rcl}
\varphi_{\text{CNP}}: \CX &\longrightarrow& \CO[G,P,\theta]\\
\xi_{p,g} &\mapsto& u_gs_p
\end{array}\]
defines a (CNP)-covariant representation of $\CX$, where $\xi_{p,g}$ denotes the representative for $u_g$ in $\CX_p$. To do so, we have to verify (1)--(4) from Definition~\ref{def:rep of prod sys} and the (CNP)-covariance condition. (1) and (2) are obvious. Using (CNP 2) to compute
\[\begin{array}{lcl}
\varphi_{\text{CNP},p}(\xi_{p,g_1})^*\varphi_{\text{CNP},p}(\xi_{p,g_2}) &=& s_p^*u_{g_1^{-1}g_2}s_p\vspace*{2mm}\\ 
&=& \chi_{\theta_p(G)}(g_1^{-1}g_2) u_{\theta_p^{-1}(g_1^{-1}g_2)}\vspace*{2mm}\\ 
&=& \varphi_{\text{CNP},1_P}(\langle \xi_{p,g_1},\xi_{p,g_2} \rangle),
\end{array}\]
we get (3). (4) follows from (CNP 1) as
\[\begin{array}{lcl}
\varphi_{\text{CNP},p}(\xi_{p,g_1})\varphi_{\text{CNP},q}(\xi_{q,g_2}) &=& u_{g_1}s_pu_{g_2}s_q\vspace*{2mm}\\ 
&=& u_{g_1\theta_p(g_2)}s_{pq}\vspace*{2mm}\\ 
&=& \varphi_{\text{CNP},pq}(\xi_{p,g_1}\alpha_p(\xi_{q,g_2})).
\end{array}\]
Thus, we are left with the (CNP)-covariance condition. But since $\CX$ is a product system of finite type, see Proposition~\ref{prop:PS for an IAD}, we only have to show that $\varphi_{\text{CNP}}$ is $(CP_F)$-covariant due to \cite{Sta2}*{Corollary 3.21}. Noting that $\varphi_p^{-1}(\CK(\CX_p)) = C^*(G)$ for all $p \in P$, we obtain
\[\begin{array}{lcl}
\psi_{\varphi_{\text{CNP}},p}(\phi_p(u_g)) &=& \psi_{\varphi_{\text{CNP}},p}\left(\sum\limits_{[h] \in G/\theta_p(G)} \Theta_{u_{gh},u_h}\right)\vspace*{2mm}\\
&=& u_g \sum\limits_{[h] \in G/\theta_p(G)} e_{h,p}\vspace*{2mm}\\
&=& u_g = \varphi_{\text{CNP},1_P}(\xi_{1_P,g}).
\end{array}\]
Thus, $\varphi_{\text{CNP}}$ is a (CNP)-covariant representation of $\CX$. By the universal property of $\CO_{\CX}$, there exists a $\ast$-homomorphism $\overline{\varphi}_{\text{CNP}}: \CO_{\CX} \longrightarrow \CO[G,P,\theta]$ such that $\overline{\varphi}_{\text{CNP}} \circ \iota_{\CO_{\CX}} = \varphi_{\text{CNP}}$. It is apparent that $\overline{\varphi}_{\text{CNP}}$ and $\varphi$ are inverse to each other, so $\varphi$ is an isomorphism.
\end{proof}


\appendix
\section{Crossed products by semidirect products}\label{sec6}
\noindent Within this section, we will establish a result about viewing a crossed product of a C*-algebra by a semidirect product of discrete, left cancellative monoids as an iterated crossed product, see Theorem~\ref{thm:cr prod by sd prod as it cr prod}. This extends the well-known result for semidirect products of locally compact groups in the discrete case, see \cite{Wil}*{Proposition~3.11}, and is essential for the proof of Corollary~\ref{cor:F cong D rtimes G-general IAD}.

For convenience, we will restrict our attention to the case of unital coefficient algebras and include the basic definitions for semigroup crossed products based on covariant pairs of representations. We refer to \cite{Lar} for a more extensive treatment of the subject.

All semigroups will be left cancellative and discrete. In the following, let $\text{Isom}(B)$ denote the semigroup of isometries in a unital C*-algebra $B$.

\begin{definition}\label{def:cov pairs}
Let $S$ be a semigroup and $A$ a unital C*-algebra with an $S$-action $\alpha$ by endomorphisms. A \emph{covariant pair} $(\pi_A,\pi_S)$ for $(A,S,\alpha)$ is given by a unital C*-algebra $B$ together with a unital $\ast$-homomorphism $\pi_A: A \longrightarrow B$ and a semigroup homomorphism $\pi_S: S \longrightarrow \text{Isom}(B)$ subject to the covariance condition:
\[\pi_S(s)\pi_A(a)\pi_S(s)^* = \pi_A(\alpha_s(a)) \text{ for all } a \in A,s \in S.\]
\end{definition}

\begin{definition}\label{def:semigroup cr prod}
Let $S$ be a semigroup and $A$ a unital C*-algebra with an $S$-action $\alpha$ by endomorphisms. The \emph{crossed product} for $(A,S,\alpha)$, denoted by $A \rtimes_\alpha S$, is the C*-algebra generated by a covariant pair $(\iota_A,\iota_S)$ which is universal in the sense that whenever $(\pi_A,\pi_S)$ is a covariant pair for $(A,S,\alpha)$, it factors through $(\iota_A,\iota_S)$. That is to say, there is a surjective $\ast$-homomorphism $\overline{\pi}: A \rtimes_\alpha S \longrightarrow C^*(\pi_A(A),\pi_S(S))$ satisfying $\pi_A = \overline{\pi} \circ \iota_A$ and $\pi_S = \overline{\pi} \circ \iota_S$.
$A \rtimes_\alpha S$ is uniquely determined up to canonical isomorphism by this universal property.
\end{definition}

\noindent This crossed product may be $0$, see \cite{Sta}*{Example 2.1(a)}. But it is known that the coefficient algebra $A$ embeds into $A \rtimes_\alpha S$ provided that $S$ acts by injective endomorphisms and is right-reversible, i.e. $Ss \cap St \neq \emptyset$ for all $s,t \in S$, see \cite{DFK}*{Lemma 5.2.1}. 

Suppose that $T$ is a semigroup which acts on another semigroup $S$ by semigroup homomorphisms $\theta_t$. Then we can form the semidirect product $S{\rtimes_\theta}T$, which is the semigroup given by $S{\times}T$ with $ax+b$-composition rule:
\[(s,t)(s',t') = (s\theta_t(s'),tt')\]
Now suppose further that $S$ and $T$ are monoids and that $\alpha$ is an action of $S{\rtimes_\theta}T$ on a unital C*-algebra $A$. Then the semigroup crossed product $A \rtimes_\alpha (S{\rtimes_\theta}T)$ is given by a unital $\ast$-homomorphism 
\[\iota_{A,S{\rtimes_\theta}T}: A \longrightarrow A \rtimes_\alpha (S{\rtimes_\theta}T)\]
and a semigroup homomorphism
\[\iota_{S{\rtimes_\theta}T}: S{\rtimes_\theta}T \longrightarrow \text{Isom}(A \rtimes_\alpha (S{\rtimes_\theta}T)).\]
Of course, we can also consider $A \rtimes_{\alpha|_S} S$ given by a unital $\ast$-homomorphism $\iota_{A,S}: A \longrightarrow A \rtimes_{\alpha|_S} S$ and a homomorphism $\iota_{S}: S \longrightarrow \text{Isom}(A \rtimes_{\alpha|_S} S)$. A natural question in this situation is whether $\alpha$ and $\theta$ give rise to a $T$-action $\tilde{\alpha}$ on $A \rtimes_{\alpha|_S} S$. The next lemma provides a positive answer for the case where $\alpha$ satisfies $\begin{array}{c} \{1_A - \alpha_{(s,1_T)}(1_A) \mid s \in S\} \subset \bigcap_{t \in T}\ker\alpha_{(1_S,t)}. \end{array}$ For the sake of readability, let $p_{(s,t)} := \iota_{A,S}(\alpha_{(s,t)}(1_A))$ for $s \in S,t \in T$ and we will simply write $p_t$ for $p_{(1_S,t)}$. We observe that the aforementioned condition is equivalent to $p_{(\theta_t(s),t)} = p_t \text{ for all } s \in S,t \in T$.

\begin{lemma}\label{lem:sgp action for the it process}
Suppose that $S$ and $T$ are monoids with a $T$-action $\theta$ on $S$ by semigroup homomorphisms. Let $\alpha$ be an action of $S \rtimes_\theta T$ on a unital C*-algebra $A$ by endomorphisms. For $t \in T$, let
\[\tilde{\alpha}_t (\iota_{A,S}(a)\iota_S(s)) := \iota_{A,S}(\alpha_{(1_S,t)}(a))\iota_S(\theta_t(s)) \text{ for } a \in A, s \in S.\]
$\tilde{\alpha}_t$ is an endomorphism from $A\rtimes_{\alpha|_S} S \longrightarrow p_t(A\rtimes_{\alpha|_S} S) p_t$ provided that
\[1_A - \alpha_{(s,1_T)}(1_A) \in \ker\alpha_{(1_S,t)} \text{ for all } s \in S.\]
In particular, if this holds for all $t \in T$, i.e.
\[\begin{array}{c} 1_A - \alpha_{(s,1_T)}(1_A) \in \bigcap_{t \in T}\ker\alpha_{(1_S,t)} \text{ for all } s \in S, \end{array}\]
then $\tilde{\alpha}$ defines an action of $T$ on $A \rtimes_{\alpha|_S} S$.
\end{lemma}
\begin{proof}
Note that $\tilde{\alpha}_t (\iota_S(s)) = \tilde{\alpha}_t (\iota_{A,S}(1_A)\iota_S(s)) = p_t\iota_S(\theta_t(s))$ is valid for all $s \in S,t\in T$ since $\iota_{A,S}$ is unital. Suppose $t \in T$ satisfies 
\[1_A - \alpha_{(s,1_T)}(1_A) \in \ker\alpha_{(1_S,t)} \text{ for all } s \in S.\]
This is equivalent to $p_{(\theta_t(s),t)} = p_t$. Hence, $p_t$ commutes with $\iota_S(\theta_t(s))$ since
\[\iota_S(\theta_t(s))p_t = \iota_S(\theta_t(s))p_t\iota_S(\theta_t(s))^*\iota_S(\theta_t(s)) = p_{(\theta_t(s),t)}\iota_S(\theta_t(s)) = p_t\iota_S(\theta_t(s)).\]
To prove that $\tilde{\alpha}_t$ is an endomorphism of $A \rtimes_{\alpha|_S} S$, we show that 
\[\left(\iota_{A,S} \circ \alpha_{(1_S,t)},p_t(\iota_S \circ \theta_t(\cdot))\right)\] 
is a covariant pair for $(A,S,\alpha|_S)$. It is then easy to see that the induced map coming from the universal property of the crossed product is precisely $\tilde{\alpha}_t$ and maps $A \rtimes_{\alpha|_S} S$ onto the corner $p_t\left(A \rtimes_{\alpha|_S} S\right)p_t$. 

$\iota_{A,S} \circ \alpha_{(1_S,t)}$ is a unital $\ast$-homomorphism from $A$ to $p_t\left(A \rtimes_{\alpha|_S} S\right)p_t$. In addition, $p_t(\iota_S \circ \theta_t(\cdot))$ maps $S$ to the isometries in $p_t\left(A \rtimes_{\alpha|_S} S\right)p_t$ because
\[(p_t\iota_S(\theta_t(s)))^*p_t\iota_S(\theta_t(s)) = \iota_S(\theta_t(s))^*p_t\iota_S(\theta_t(s)) = \iota_S(\theta_t(s))^*\iota_S(\theta_t(s))p_t = p_t.\]
This map turns out to be a semigroup homomorphism as
\[p_t \iota_S(\theta_t(s_1))p_t\iota_S(\theta_t(s_2)) = p_t^2\iota_S(\theta_t(s_1))\iota_S(\theta_t(s_2)) = p_t\iota_S(\theta_t(s_1s_2)).\] 
Finally, for $a \in A$ and $s \in S$, we compute
\[\begin{array}{lcl}
p_t\iota_S(\theta_t(s)) \iota_{A,S}(\alpha_{(1_S,t)}(a)) (p_t\iota_S(\theta_t(s)))^* &=& p_t\iota_{A,S}(\alpha_{(\theta_t(s),t)}(a))p_t\\
&=& \iota_{A,S}(\alpha_{(1_S,t)}(\alpha_{(s,1_T)}(a)).
\end{array}\]
Thus, $\left(\iota_{A,S} \circ \alpha_{(1_S,t)},p_t(\iota_S \circ \theta_t(\cdot))\right)$ forms a covariant pair for $(A,S,\alpha|_S)$. In particular, the induced map $\tilde{\alpha}_t$ is an endomorphism of $A \rtimes_{\alpha|_S} S$. 

Conversely, assume that $\tilde{\alpha}_t$ defines an endomorphism of $A \rtimes_{\alpha|_S} S$. Then $(\tilde{\alpha}_t \circ \iota_{A,S},\tilde{\alpha}_t \circ \iota_S)$ forms a covariant pair for $(A,\alpha|_S,S)$ mapping $A$ and $S$ to the C*-algebra $B:=\tilde{\alpha}_t(A\rtimes_{\alpha|_S}S)$. Note that the unit inside this C*-algebra is $p_t$. In particular, we have a semigroup homomorphism $\tilde{\alpha}_t \circ \iota_S:S \longrightarrow \text{Isom}(B)$. This forces 
\[p_t = \tilde{\alpha}_t(\iota_S(s))^*\tilde{\alpha}_t(\iota_S(s)) = \iota_S(\theta_t(s))^*p_t\iota_S(\theta_t(s)) = p_{(\theta_t(s),t)}\]
for all $s \in S$, which is equivalent to 
\[\{1_A - \alpha_{(s,1_T)}(1_A) \mid s \in S\} \subset \ker\alpha_{(1_S,t)}.\]
\noindent Since $\alpha|_T$ and $\theta$ are semigroup homomorphisms, $\tilde{\alpha}$ defines an action of $T$ on $A \rtimes_{\alpha|_S} S$ provided that the imposed condition holds for every $t \in T$. 
\end{proof}

\begin{remark}\label{rem:cr prod dec lem - nec of cond}
It would be interesting to know whether the condition from Lemma~\ref{lem:sgp action for the it process} is actually necessary. This would be the case if $p_t \leq p_{(\theta_t(s),1_T)}$ was true for $s,t \in S$. Note that $p_{(\theta_t(s),t)} \leq p_t$ and $p_{(\theta_t(s),t)} \leq p_{(\theta_t(s),1_T)}$.
\end{remark}

\noindent Given the hypotheses of Lemma~\ref{lem:sgp action for the it process} are satisfied, $\tilde{\alpha}$ gives rise to an iterated semigroup crossed product $\left(A \rtimes_{\alpha|_S} S\right) \rtimes_{\tilde{\alpha}} T$ and it is a natural task to relate this crossed product to $A \rtimes_\alpha (S{\rtimes_\theta}T)$. The next result shows that indeed, this decomposition procedure recovers the original crossed product.

\begin{theorem}\label{thm:cr prod by sd prod as it cr prod}
Suppose $S$ and $T$ are monoids together with a $T$-action $\theta$  on $S$ by semigroup homomorphisms, and an action $\alpha$ of $S{\rtimes_\theta}T$ on a unital C*-algebra $A$ by endomorphisms. If 
\[\begin{array}{c} \{1_A - \alpha_{(s,1_T)}(1_A) \mid s \in S\} \subset \bigcap_{t \in T}\ker\alpha_{(1_S,t)} \end{array}\]
holds true, then there is a canonical isomorphism
\[\begin{array}{rcl}
A \rtimes_\alpha \left(S{\rtimes_\theta}T\right) &\stackrel{\pi}{\longrightarrow}& \left(A \rtimes_{{\alpha}|_S} S\right) \rtimes_{\tilde{\alpha}} T,\vspace*{2mm}\\
\iota_{A,S{\rtimes_\theta}T}(a) &\mapsto& \iota_{A \rtimes S} \circ \iota_{A,S}(a)\vspace*{2mm}\\
\iota_{S{\rtimes_\theta}T}(s,t) &\mapsto& (\iota_{A \rtimes S} \circ \iota_{S})(s) \iota_T(t)\vspace*{2mm}\\
\end{array}\]  
where $\tilde{\alpha}$ is given by $\tilde{\alpha}_t(\iota_{A,S}(a)\iota_S(s)) = \iota_A(\alpha_{(1_S,t)}(a))\iota_S(\theta_t(s))$.
\end{theorem}

\begin{proof}
Recall that $(\iota_{A,S{\rtimes_\theta}T},\iota_{S{\rtimes_\theta}T}), (\iota_{A,S},\iota_S)$ and $(\iota_{A \rtimes S},\iota_T)$ denote the universal covariant pairs for $(A,S{\rtimes_\theta}T,\alpha)$, $(A,S,\alpha|_S)$ and $(A \rtimes_{\alpha|_S} S,T,\tilde{\alpha})$, respectively. The strategy is governed by the following claims: 
\begin{enumerate}[1)]
\item $(\iota_{A \rtimes S} \circ \iota_{A,S}, (\iota_{A \rtimes S} \circ \iota_{S}) \times \iota_T)$ forms a covariant pair for $(A,S{\rtimes_\theta}T,\alpha)$.
\item $(\iota_{A,S{\rtimes_\theta}T} \times \iota_{S{\rtimes_\theta}T}|_S,\iota_{S{\rtimes_\theta}T}|_T)$ forms a covariant pair for $(A \rtimes_{\alpha|_S} S,T,\tilde{\alpha})$. 
\end{enumerate}
If we assume 1) and 2), then 1) and the universal property of $A \rtimes_\alpha (S{\rtimes_\theta}T)$ give a $\ast$-homomorphism 
\[\begin{array}{rcl}
A \rtimes_\alpha (S{\rtimes_\theta}T) &\stackrel{\pi}{\onto}& (A \rtimes_{{\alpha}|_S} S) \rtimes_{\tilde{\alpha}} T\\
\iota_{A,S{\rtimes_\theta}T}(a) &\mapsto& \iota_{A \rtimes S} \circ \iota_{A,S}(a)\\
\iota_{S{\rtimes_\theta}T}(s,t) &\mapsto& (\iota_{A \rtimes S} \circ \iota_{S})(s) \iota_T(t)
\end{array}\]
Since $S$ and $T$ both have an identity, the induced map equals $\pi$. Note that the pair from 2) is the natural candidate to provide an inverse for $\pi$. Indeed, if 2) is valid, then the two induced $\ast$-homomorphisms are mutually inverse on the standard generators of the C*-algebras on both sides. Thus it remains to establish 1) and 2).

For step 1), note that $\iota_{A \rtimes S} \circ \iota_{A,S}$ is a unital $\ast$-homomorphism and $\iota_{A \rtimes S} \circ \iota_{S}$ defines a semigroup homomorphism from $S$ to the isometries in $(A \rtimes_{{\alpha}|_S} S) \rtimes_{\tilde{\alpha}} T$. The covariance condition for $(T,\tilde{\alpha})$ yields 
\[\iota_T(t)\iota_{A \rtimes S} \circ \iota_{S}(s) = \tilde{\alpha}(\iota_{A \rtimes S} \circ \iota_{S}(s))\iota_T(t) = \iota_{A \rtimes S} \circ \iota_{S}(\theta_t(s))\iota_T(t).\]   
Therefore, $(\iota_{A \rtimes S} \circ \iota_{S}) \times \iota_T$ is well-behaved with respect to the semidirect product structure on $S \times T$ coming from $\theta$, so we get a semigroup homomorphism $(\iota_{A \rtimes S} \circ \iota_{S}) \times \iota_T: S{\rtimes_\theta}T \longrightarrow \text{Isom}((A \rtimes_{{\alpha}|_S} S) \rtimes_{\tilde{\alpha}} T)$. Now let $a \in A, s \in S$ and $t \in T$. Then we compute
\[\begin{array}{l}
((\iota_{A \rtimes S} \circ \iota_{S}) \times \iota_T)(s,t) \iota_{A \rtimes S} \circ \iota_{A,S}(a) ((\iota_{A \rtimes S} \circ \iota_{S}) \times \iota_T)(s,t)^*\vspace*{2mm}\\ 
\hspace*{8mm}= \iota_{A \rtimes S} \circ \iota_{S}(s)\iota_T(t) \iota_{A \rtimes S} \circ \iota_{A,S}(a) \iota_T(t)^*\iota_{A \rtimes S} \circ \iota_{S}(s)^*\vspace*{2mm}\\
\hspace*{8mm}= \iota_{A \rtimes S} \circ \iota_{S}(s) \iota_{A \rtimes S} \circ \iota_{A,S}(\alpha_{(1_S,t)}(a)) \iota_{A \rtimes S} \circ \iota_{S}(s)^*\vspace*{2mm}\\
\hspace*{8mm}= \iota_{A \rtimes S} \circ \iota_{A,S}(\alpha_{(s,1_T)(1_S,t)}(a))\vspace*{2mm}\\
\hspace*{8mm}= \iota_{A \rtimes S} \circ \iota_{A,S}(\alpha_{(s,t)}(a)),
\end{array}\]
which completes 1). For part 2), we remark that $(\iota_{A,S{\rtimes_\theta}T},\iota_{S{\rtimes_\theta}T}|_S)$ is a covariant pair for $(A,S,\alpha|_S)$. Since $\iota_{A,S{\rtimes_\theta}T}$ and $\iota_{A,S}$ are unital, the induced map is unital as well. Moreover, $\iota_{S{\rtimes_\theta}T}|_T$ is a semigroup homomorphism mapping $T$ to the isometries in $A \rtimes_\alpha (S{\rtimes_\theta}T)$. Thus, we are left with the covariance condition. Note that it suffices to check the covariance condition on the standard generators of $A \rtimes_{\alpha|_S} S$. For $a \in A, s \in S$ and $t \in T$, we get
\[\begin{array}{l}
\hspace*{-1mm}\iota_{S{\rtimes_\theta}T}(1_S,t) \iota_{A,S\rtimes_\theta T}(a)\iota_{S{\rtimes_\theta}T}(s,1_T) \iota_{S{\rtimes_\theta}T}(1_S,t)^* \vspace*{2mm}\\
=\iota_{S{\rtimes_\theta}T}(1_S,t) \iota_{A,S\rtimes_\theta T}(a)\iota_{S{\rtimes_\theta}T}(1_S,t)^*\iota_{S{\rtimes_\theta}T}(1_S,t)\iota_{S{\rtimes_\theta}T}(s,1_T) \iota_{S{\rtimes_\theta}T}(1_S,t)^* \vspace*{2mm}\\ 
= \iota_{A,S\rtimes_\theta T}(\alpha_{(1_S,t)}(a))\iota_{S{\rtimes_\theta}T}(\theta_t(s),1_T)p_t \vspace*{2mm}\\
= \iota_{A,S\rtimes_\theta T}(\alpha_{(1_S,t)}(a))\iota_{S{\rtimes_\theta}T}(\theta_t(s),1_T)\vspace*{2mm}\\
=\tilde{\alpha}_t(\iota_{A,S\rtimes_\theta T}(a)\iota_{S{\rtimes_\theta}T}(s,1_T)).
\end{array}\]
Hence 1) and 2) are both valid, so the proof is complete.
\end{proof}


\begin{remark}\label{rem:ext to proper endomorphisms for non-unital coeff algs}~
\begin{enumerate}[a)]
\item It is conceivable that Theorem~\ref{thm:cr prod by sd prod as it cr prod} extends to the setting where $A$ need not be unital, representations are non-degenerate and $\alpha$ is extendible, see \cite{Lar} for more information on these conditions.
\item The condition $p_{(\theta_t(s),t)} = p_t$ for all $s \in S$ and $t \in T$ is satisfied if $\alpha|_S$ is unital because $\alpha_{(\theta_t(s),t)}(1_A) = \alpha_{(1_S,t)}(\alpha_{(s,1_T)}(1_A)) = \alpha_{(1_S,t)}(1_A)$. In fact, if $\alpha|_T$ consists of injective endomorphisms, then $p_{(\theta_t(s),t)} = p_t$ holds if and only if $\alpha|_S$ is unital. In particular, $p_{(\theta_t(s),t)} = p_t$ is satisfied whenever $S$ is a group.
\end{enumerate}
\end{remark}

\end{document}